\documentclass[11pt]{amsart} 
\usepackage{amscd,amssymb,amsxtra}
\usepackage[mathscr]{eucal}
\usepackage{wrapfig}
\usepackage{graphicx}
\usepackage{comment}
\usepackage{color}
\usepackage{enumerate}
\usepackage{mathtools}
\makeatletter
\let\@secnumfont\bfseries
\def\section{\@startsection{section}{1}%
  \z@{4\linespacing\@plus\linespacing}{\linespacing}%
  {\bfseries\centering}}
\def\introsection{\@startsection{section}{1}%
  \z@{3\linespacing\@plus\linespacing}{\linespacing}%
  {\bfseries\centering}}
\def\subsection{\@startsection{subsection}{2}%
   \z@{1.25\linespacing\@plus.7\linespacing}{.5\linespacing}%
   {\normalfont\bfseries}}
\def\subsectionsinline{\def\subsection{\@startsection{subsection}{2}%
  \z@{1\linespacing\@plus.7\linespacing}{-.5em}%
  {\normalfont\bfseries}}}

\makeatother

\theoremstyle{definition}
\newtheorem{definition}[equation]{Definition}
\newtheorem{example}[equation]{Example}

\newtheorem*{definition*}{Definition}
\newtheorem*{example*}{Example}
\newtheorem*{problem*}{Problem}
\newtheorem*{exercise*}{Exercise}
\newtheorem*{question*}{Question}
\newtheorem*{construction*}{Construction}

\theoremstyle{remark}

\newtheorem{remark}[equation]{Remark}

\newtheorem*{note*}{Note}
\newtheorem*{notation*}{Notation}
\newtheorem*{remark*}{Remark}

\theoremstyle{plain}
\newtheorem{theorem}[equation]{Theorem}
\newtheorem{corollary}[equation]{Corollary}
\newtheorem{lemma}[equation]{Lemma}

\newtheorem*{theorem*}{Theorem}
\newtheorem*{corollary*}{Corollary}
\newtheorem*{lemma*}{Lemma}
\newtheorem*{proposition*}{Proposition}
\newtheorem*{conjecture*}{Conjecture}
\newtheorem*{claim*}{Claim}
\newtheorem*{proposal*}{Proposal}
\newtheorem*{conclusion*}{Conclusion}
\newtheorem*{hypothesis*}{Hypothesis}

\numberwithin{equation}{section}

\definecolor{refkey}{rgb}{0,.6,.4}

\renewcommand{\:}{\colon}

\DeclareMathOperator{\Aut}{Aut}
\newcommand{\CC}{{\mathbb C}}
\newcommand{\CP}{{\mathbb C\mathbb P}}

\DeclareMathOperator{\End}{End}

\DeclareMathOperator{\Hom}{Hom}
\DeclareMathOperator{\id}{id}
\DeclareMathOperator{\Map}{Map}

\DeclareMathOperator{\pt}{pt}

\newcommand{\RR}{{\mathbb R}}

\newcommand{\ZZ}{{\mathbb Z}}

\newcommand{\chiup}{\raise.5ex\hbox{$\chi$}}
\newcommand{\cir}{S^1}

\newcommand{\inv}{^{-1}}
\newcommand{\mstrut}{^{\vphantom{1*\prime y\vee M}}}

\newcommand{\res}[1]{\negmedspace\bigm|\mstrut_{#1}}
\newcommand{\temsquare}{\raise3.5pt\hbox{\boxed{ }}}

\newcommand{\zmod}[1]{\ZZ/#1\ZZ}

\newcommand{\zt}{\zmod2}

\usepackage[all,2cell,dvips]{xy}\renewcommand{\cir}{\ensuremath{S^1}}
\UseAllTwocells

\definecolor{refkey}{rgb}{0,.8,.2}\definecolor{labelkey}{rgb}{1,0,0}

\DeclareMathOperator{\Ab}{Ab}
\DeclareMathOperator{\Bord}{Bord}
\DeclareMathOperator{\Cat}{Cat}
\DeclareMathOperator{\Euler}{Euler}
\DeclareMathOperator{\Hilb}{Hilb}
\DeclareMathOperator{\Imm}{Im}
\DeclareMathOperator{\Riem}{Riem}
\DeclareMathOperator{\Top}{Top}
\DeclareMathOperator{\Vect}{Vect}
\DeclareMathOperator{\coev}{coev}
\DeclareMathOperator{\ev}{ev}
\DeclareMathOperator{\fr}{fr}
\newcommand{\CAlg}{Alg_{\CC}}
\newcommand{\CCat}{\Cat_{\CC}}
\newcommand{\CVect}{\Vect_{\CC}}
\newcommand{\bX}{\partial X}
\newcommand{\bordG}[1]{\Bord^{G}_{#1}}
\newcommand{\bordfr}[1]{\Bord^{\fr}_{#1}}
\newcommand{\bordn}{\Bord_{\langle 0\cdots n  \rangle}}
\newcommand{\bordo}[1]{\Bord^{O}_{#1}}
\newcommand{\bordso}[1]{\Bord^{SO}_{#1}}
\newcommand{\bordt}{\Bord_{\langle 0,1,2  \rangle}}
\newcommand{\bord}[1]{\Bord_{#1}}
\newcommand{\bsD}{|\mathcal{D}|}
\newcommand{\fld}[1]{\mathscr{F}_{#1}}
\newcommand{\gpd}{/\!/} 
\newcommand{\hF}{\widehat{F}}
\newcommand{\infn}{(\infty ,n)}
\newcommand{\obs}{\mathcal{O}}
\newcommand{\pp}{\pt_+}
\newcommand{\sCfdt}{(\sCfd)^{\sim}}
\newcommand{\sCfd}{\sC^{\textnormal{fd}}}

\newcommand{\sC}{\mathcal{C}}
\newcommand{\sDt}{\sD^{\sim}}
\newcommand{\sD}{\mathcal{D}}
\newcommand{\sH}{\mathscr{H}}
\newcommand{\sX}{\mathscr{X}}
\newcommand{\tF}{\widetilde{F}}
\newcommand{\tbordfr}{\Bord_{\langle n-1,n  \rangle}^{\fr}}
\newcommand{\tbordr}{\Bord_{\langle n-1,n  \rangle}^{\Riem}}
\newcommand{\tbordso}{\Bord_{\langle n-1,n  \rangle}^{SO}}
\newcommand{\tbord}{\Bord_{\langle n-1,n  \rangle}}
\renewcommand{\AA}{\mathbb{A}}

\begin{document}

\abovedisplayskip18pt plus4.5pt minus9pt
\belowdisplayskip \abovedisplayskip
\abovedisplayshortskip0pt plus4.5pt
\belowdisplayshortskip10.5pt plus4.5pt minus6pt
\baselineskip=15 truept
\marginparwidth=55pt

\renewcommand{\labelenumi}{\textnormal{(\roman{enumi})}}



 \title[The cobordism hypothesis]{The Cobordism Hypothesis} 
 \author[D. S. Freed]{Daniel S.~Freed}
 \thanks{The work of D.S.F. is supported by the National Science Foundation
under grant DMS-0603964}
 \address{The University of Texas at Austin \\ Mathematics Department RLM
8.100 \\ 2515 Speedway Stop C1200\\ Austin, TX 78712-1202}
 \email{dafr@math.utexas.edu}
 \date{July 5, 2012}
 \begin{abstract} 
 In this expository paper we introduce extended topological quantum field
theories and the cobordism hypothesis.
 \end{abstract}
\maketitle

   \section{Introduction}\label{sec:1}

The \emph{cobordism hypothesis} was conjectured by Baez-Dolan~\cite{BD} in
the mid 1990s.  It has now been proved by Hopkins-Lurie in dimension two and
by Lurie in higher dimensions.  There are many complicated foundational
issues which lie behind the definitions and the proof, and only a detailed
sketch~\cite{L1} has appeared so far.\footnote{Nonetheless, we use `theorem'
and its synonyms in this manuscript.  The foundations are rapidly being
filled in and alternative proofs have also been carried out, though none has
yet appeared in print.}  The history of the Baez-Dolan conjecture goes most
directly through quantum field theory and its adaptation to low-dimensional
topology.  Yet in retrospect it is a theorem about the structure of manifolds
in all dimensions, and at the core of the proof lies Morse theory.  Hence
there are two routes to the cobordism hypothesis: algebraic topology and
quantum field theory.

Consider the abelian group~$\Omega ^{SO}_0$ generated by compact oriented
0-dimensional manifolds, that is, finite sets~$Y$ of points each labeled
with~$+$ or~$-$.  The group operation is disjoint union.  We deem~$Y_0$
equivalent to~$Y_1$ if there is a compact oriented 1-manifold~$X$ with
oriented boundary $Y_1\amalg -Y_0$.  Then a basic theorem in differential
topology~\cite[Appendix]{Mi1} asserts that $\Omega ^{SO}_0$~is the free
abelian group with a single generator, the positively oriented
point~$\pp$.\footnote{Two important remarks: (1)\ we can replace orientations
with framings; (2)\ for unoriented manifolds the group~$\Omega ^{O}_0$ is not
free on one generator, but rather there is a relation and $\Omega ^O_0\cong
\zt$.}  This result is the cornerstone of smooth intersection theory.  From
the point of view of algebraic topology the cobordism hypothesis is a similar
statement about a more ornate structure built from smooth manifolds.  The
simplest version is for framed manifolds.  The language is off-putting if
unfamiliar, and it will be explained in due course.

  \begin{theorem}[Cobordism hypothesis: heuristic algebro-topological
  version]\label{thm:13}
 For $n\ge1$, $\bordfr n$~is the free symmetric monoidal $(\infty
,n)$-category with duals generated by~$\pp$.
  \end{theorem}

\noindent
 The `Bord' in~$\bordfr n$ stands for `bordism',\footnote{`Bordism' replaces
the older `cobordism', as bordism is part of homology whereas cobordism is
part of cohomology~\cite{A1}.} and $\pp$~is now the point with the standard
framing.  $\bordfr n$~is an elaborate algebraic gadget which encodes
$n$-framed manifolds with corners of dimensions~$\le n$ and tracks gluings
and disjoint unions.  One of our goals is to motivate this elaborate
algebraic structure.

An \emph{extended topological field theory} is a representation of the
bordism category, i.e., a homomorphism $F\:\bordfr n\to\sC$.  The
codomain~$\sC$ is a symmetric monoidal $\infn$-category, typically linear in
nature.  In important examples $F$~assigns a complex number to every closed
$n$-manifold and a complex vector space to every closed $(n-1)$-manifold.

  \begin{theorem}[Cobordism hypothesis: weak quantum field theory
  version]\label{thm:14} 
 A homomorphism\\$F\:\bordfr n\to\sC$ is determined by~$F(\pp)$. 
  \end{theorem}

\noindent
 The object~$F(\pp)\in \sC$ satisfies stringent finiteness conditions
expressed in terms of dualities, and the real power of the cobordism
hypothesis is an existence statement: if $x\in \sC $ is
$n$-\emph{dualizable}, then there exists a topological field theory~$F$
with~$F(\pp)=x$.  Precise statements of the cobordism hypothesis appear
in~\S\ref{sec:6}. 
 
Our plan is to build up gradually to the categorical complexities inherent in
extended field theories and the cobordism hypothesis.  So in the next two
sections we take strolls along the two routes to the cobordism hypothesis:
algebraic topology~(\S\ref{sec:2}) and quantum field theory~(\S\ref{sec:3}).
Section~\ref{sec:4} is an extended introduction to non-extended topological
field theory.  The simple examples discussed there only hint at the power of
this circle of ideas.  In~\S\ref{sec:5} we turn to extended field theories
and so also to higher categories.  The cobordism hypothesis is the subject
of~\S\ref{sec:6}, where we state a complete version in Theorem~\ref{thm:26}.
The cobordism hypothesis connects in exciting ways to other parts of
topology, geometry, and representation theory as well as to some contemporary
ideas in quantum field theory.  A few of these are highlighted
in~\S\ref{sec:7}.

The manuscript~\cite{L1} has leisurely introductions to higher categorical
ideas and to the setting of the cobordism hypothesis, in addition to a
detailed sketch of the proof and applications.  The original paper~\cite{BD}
is another excellent source of expository material.  Additional recent
expositions are available in~\cite{L3}, \cite{Te1}.  We have endeavored to
complement these expositions rather than duplicate them.  I warmly thank
David Ben-Zvi, Andrew Blumberg, Lee Cohn, Tim Perutz, Ulrike Tillmann, and
the referee for their comments and suggestions.

   \section{Algebraic topology}\label{sec:2}

The most basic maneuvers in algebraic topology extract algebra from spaces.
For example, to a topological space~ $X$ we associate a sequence of abelian
groups ~$\bigl\{H_q(X)\bigr\}$.  There are several constructions of these
\emph{homology groups}, but for nice spaces they are all
equivalent~\cite{Sp}.  The homology construction begins to have teeth only
when we tell how homology varies with~ $X$.  One elementary assertion is that
if $X\simeq Y$ are homeomorphic spaces, then the homology groups are
isomorphic.  Thus numerical invariants of homology groups, such as the rank,
are homeomorphism invariants of topological spaces: Betti numbers.  But it is
much more powerful to remember the isomorphisms of homology groups associated
to homeomorphisms, and indeed the homomorphisms associated to arbitrary
continuous maps.  This is naturally encoded in the algebraic structure of a
\emph{category}.  Here is an informal definition; see standard texts
(e.g.~\cite{Mc}) for details.

  \begin{definition}[]\label{thm:1}
 A \emph{category}~$\sC$ consists of a set\footnote{We do not worry about
technicalities of set theory in this expository paper.}~$C_0$ of
objects~$\{x\}$, a set~$C_1$ of morphisms $\{f\:x\to y\}$, identity elements
$\{id_x\:x\to x\}$, and an associative composition law $f,g\longmapsto g\circ
f$ for morphisms $x\xrightarrow{f}y$ and $y\xrightarrow{g}z$.  If $\sC,\sD$
are categories then a \emph{homomorphism}\footnote{The word `functor' is
usually employed here, but 'homomorphism' is more consistent with standard
usage elsewhere in algebra.}  $F\:\sC\to\sD$ is a pair~$(F_0,F_1)$ of maps of
sets $F_i\:\sC_i\to \sD_i$ which preserves compositions.
  \end{definition}

\noindent
 More formally, there are source and target maps $C_1\to C_0$, identity
elements are defined by a map $C_0\to C_1$, and composition is a map from a
subset of $C_1\times C_1$ to~$C_1$---the subset consists of pairs of
morphisms for which the target of the first equals the source of the second.
A homomorphism also preserves the source and target maps.  Topological spaces
comprise the objects of a category~$\Top$ whose morphisms are continuous
maps; abelian groups comprise the objects of a category~$\Ab$ whose morphisms
are group homomorphisms.  Some basic properties of homology groups are
summarized by the statement that
  \begin{equation}\label{eq:1}
     H_q\:(\Top,\amalg )\longrightarrow (\Ab ,\oplus )
  \end{equation}
is a homomorphism.  We explain the~`$\amalg $' and~`$\oplus $' in the next
paragraph.

The homomorphism property does not nearly characterize homology, and we can
encode many more properties via extra structure on~$\Top$ and~$\Ab$.  We
single out one here, an additional operation on objects and morphisms.  If
$X_1,X_2$ are topological spaces there is a new space $X_1\amalg X_2$, the
\emph{disjoint union}.  The operation $X_1,X_2\mapsto X_1\amalg X_2$ has
properties analogous to a commutative, associative composition law on a set.
For example, the empty set~$\emptyset $ is an identity for disjoint union in
the sense that $\emptyset \amalg X$ is canonically identified with~$X$ for
all topological spaces~$X$.  Furthermore, if $f_i\:X_i\to Y_i,\;i=1,2$ are
continuous maps, there is an induced continuous map $f_1\amalg f_2\:X_1\amalg
X_2\to Y_1\amalg Y_2$ on the disjoint union.  An operation on a category with
these properties is called a \emph{symmetric monoidal structure}, in this
case on the category~$\Top$.  Similarly, the category~$\Ab$ of abelian groups
has a symmetric monoidal structure given by direct sum: $A_1,A_2\to A_1\oplus
A_2$.  The homology maps~\eqref{eq:1} are homomorphisms of symmetric monoidal
categories: there is a canonical identification of~$H_q(X_1\amalg X_2)$ with
$H_q(X_1)\oplus H_q(X_2)$.

  \begin{remark}[]\label{thm:2}
 Homology is \emph{classical} in that disjoint unions map to \emph{direct
sums}.  We will see that a characteristic property of \emph{quantum} systems
is that disjoint unions map to \emph{tensor products}.  The passage from
classical to quantum is (poetically) a passage from addition to
multiplication, a kind of exponentiation.
  \end{remark}

Our interest here is not in all topological spaces, but rather in smooth
manifolds.  Fix a positive integer~$n$.

  \begin{definition}[]\label{thm:3}
 Let $Y_0,Y_1$~be smooth compact $(n-1)$-dimensional manifolds without
boundary.  A \emph{bordism from~$Y_0$ to~$Y_1$} is a compact $n$-dimensional
manifold~$X$ with boundary, a decomposition $\bX=\bX_0\amalg \bX_1$, and
diffeomorphisms $Y_i\to\bX_i,\,i=1,2$. 
  \end{definition}
\begin{figure}[h!]
 \centering\includegraphics[scale=.7]{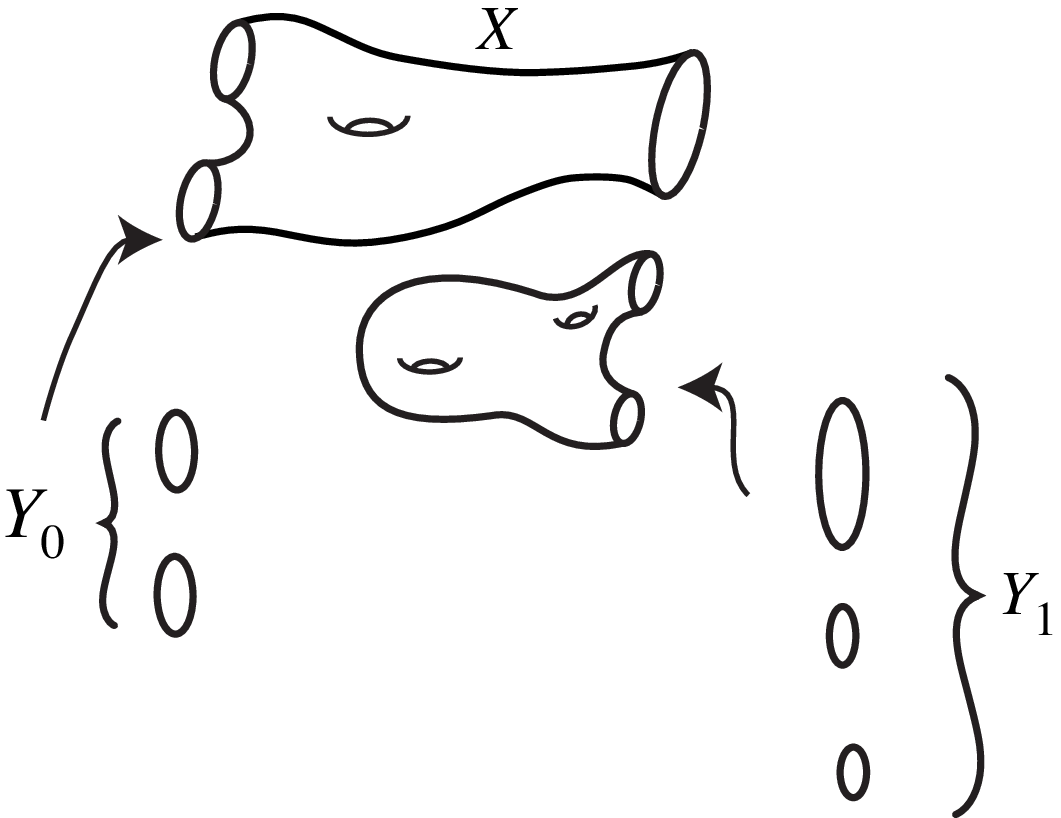}
\caption{A bordism $X\:Y_0\to Y_1$}  \label{fig:1}
 \end{figure}

\noindent
 Figure~1 depicts an example which emphasizes that manifolds need not be
connected.  The empty set~$\emptyset $ is a manifold of any dimension.  So a
\emph{closed} $n$-manifold---that is, a compact manifold without
boundary---is a bordism from~$\emptyset ^{n-1}$ to~$\emptyset ^{n-1}$.   Note
also that the disjoint union of smooth manifolds is a smooth manifold, and
the disjoint union of bordisms is a bordism.
 
To turn bordism into algebra we observe that bordism defines an equivalence
relation: closed $(n-1)$-manifolds $Y_0,Y_1$ are bordant if there exists a
bordism from~$Y_0$ to~$Y_1$.  (Observe that to prove transitivity it is
convenient to modify Definition~\ref{thm:3} so that boundary identifications
are between the manifolds $[0,1 )\times Y_0$, $(-1,0]\times Y_1$ and open
collar neighborhoods of~$\bX_0$, $\bX_1$: smooth functions glue nicely on
open sets.)  Disjoint union defines an abelian group structure on the
set~$\Omega^O _{n-1}$ of equivalence classes.  For example, $\Omega ^O_0\cong
\zt$ is generated by a single point.  Twice a point is the disjoint union of
two points, and as two points bound a closed interval, two points are bordant
to the empty 0-manifold.  Life is more interesting when we consider manifolds
with extra topological structure.  For example, there are bordism
groups~$\Omega _q^{SO}$ of oriented manifolds.  An orientation on a
0-manifold consisting of a single point is a choice of~$+$ or~$-$.  Then
$\Omega _0^{SO}\cong \ZZ$ by the map which sends a finite set of oriented
points to the number of positive points minus the number of negative points.
This is a foundational result in differential topology which enables oriented
counts in intersection theory~\cite{Mi1}.  Another interesting structure is a
\emph{stable framing}.  It arises in the Pontrjagin-Thom construction.  Let
$f\:S^{q+N}\to S^N$ be a smooth map.  By Sard's theorem there is a regular
value~$p\in S^N$, whence $M:=f\inv (p)\subset S^{q+N}$ is a smooth
$q$-dimensional submanifold.  Also, a basis of~$T_pS^N$ pulls back under~$f$
to a global framing of the normal bundle to~$M$ in~$S^N$.  \begin{figure}[h]
 \centering
 \includegraphics[scale=.5]{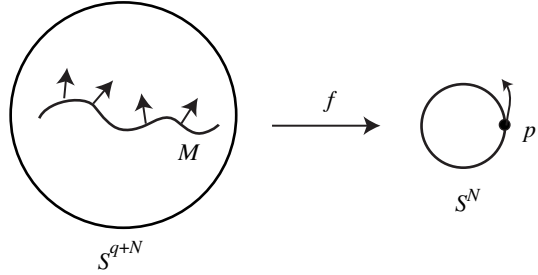}
 \caption{The Pontrjagin-Thom construction}\label{fig:2} 
 \end{figure}
If we deform~$p$ to another regular value, then the framed manifold~$M$
undergoes a bordism.  The same is true if $f$~deforms to a smoothly homotopic
map.  The precise correspondence works in the stable limit~$N\to\infty $: the
stably framed bordism group~$\Omega _q^{\fr}$ is isomorphic to the stable
homotopy group of the sphere $\lim\limits_{N\to\infty }\pi _{q+N}(S^N)$.
This is the most basic link between bordism and homotopy theory.
 
Bordism has a long history in algebraic topology.  By~1950 it
appears\footnote{According to~\cite[\S6]{May} a 1950~ Russian paper of
Pontrjagin contains bordism groups; see~\cite{P} for a later account.
Thom~\cite{T} also cites work of Rohlin relevant to computations of bordism
in low dimensions, but I do not know if Rohlin phrased them in terms of
bordism groups.} that Pontrjagin had defined abelian groups based on the
notion of a bordism, though it was Thom~\cite{T} who made the first
systematic computations of bordism groups using homotopy theory.  There are
many variations according to the type of manifold: oriented, spin, framed,
etc.  Theory and computation of bordism groups were an important part of
algebraic topology in the 1950s and 1960s, and they found applications in
other parts of topology and geometry.  For example, Hirzebruch's 1954~proof
of the Riemann-Roch theorem was based on bordism computations, as was the
first proof of the Atiyah-Singer index theorem~\cite{Pa} in~1963.
 
The bordism group of $d$-dimensional manifolds arises when
$(d+1)$-dimensional bordisms are used to define an equivalence relation.
Disjoint union of $d$-manifolds gives the abelian group structure.  One
lesson from classical algebraic topology is that the passage from Betti
numbers to homology groups is very fruitful.  The analog here is to
\emph{track} bordisms between closed manifold, not merely to observe their
existence---in our ``categorified'' world we \emph{encode} the bordism as a
map.  Segal~\cite{Se2} introduced a bordism category of Riemann surfaces in
his axiomatization of 2-dimensional conformal field theory, which inspired
Atiyah~\cite{A2} to axiomatize topological field theories in any dimensions
using bordism categories of smooth manifolds with no continuous geometric
structure (such as a metric or conformal structure).
Tillmann~\cite{Til1,Til2} observed that the classifying space of the bordism
category, which has the abelian group-like operation of disjoint union, is a
\emph{spectrum} in the sense of stable homotopy theory.  Together with
Madsen~\cite{MT} they conjecturally identify the classifying spectrum of an
enriched bordism category---a step towards the $\infty $-categories we meet
in~\S\ref{sec:5}---and show that their conjecture implies Mumford's
conjecture~\cite{Mu} about the rational cohomology of the mapping class
group.  The Madsen-Tillmann conjecture was subsequently proved in~\cite{MW}
and is now known as the Madsen-Weiss theorem.  The relation with the spectra
Thom used to compute bordism groups is elucidated in~\cite[\S3]{GMTW}, where
another proof is given.

For now we restrict to manifolds with boundary---no corners---and so organize
closed $(n-1)$-manifolds into a symmetric monoidal category which refines the
abelian group~$\Omega _{n-1}$.

  \begin{definition}[]\label{thm:4}
 $\tbord$~is the symmetric monoidal category whose objects are compact
$(n-1)$-manifolds and in which a morphism $X\:Y_0\to Y_1$ is a bordism
from~$Y_0$ to~$Y_1$, up to diffeomorphism.  The monoidal structure is
disjoint union. 
  \end{definition}

\noindent
 So now a bordism is a map---a morphism in a category---and the boundary is
divided into incoming (domain) boundary components and outgoing (codomain)
boundary components.  Heuristically, we say there is an ``arrow of time'', at
least at the boundary; in pictures we draw a global arrow of time.
Composition (Figure~\ref{fig:3}) is defined by gluing bordisms.
\begin{figure}[h]
 \centering
 \includegraphics[scale=.7]{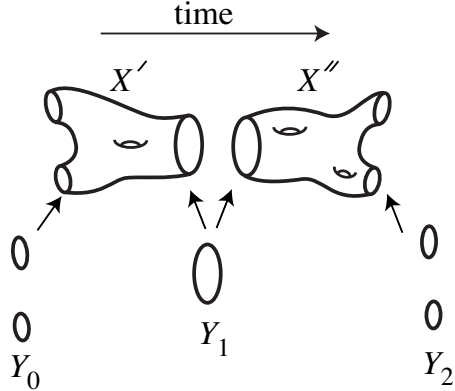}
 \caption{Composition of bordisms}\label{fig:3}
 \end{figure}
We identify diffeomorphic bordisms---the diffeomorphism must commute with the
boundary identifications---and so obtain a strictly associative composition
law.  The identity morphism $Y\to Y$ is the cylinder $[0,1]\times Y$ with
obvious boundary identifications.  There are variants $\tbordso$ and
$\tbordfr$ for oriented and framed manifolds, but with one important change:
in $\tbordfr$ the morphisms~$X$ carry framings of the {tangent} bundle
(not stabilized) and the objects~$Y$ carry framings of~$(1)\oplus TY$, where
`$(1)$'~here denotes the trivial real line bundle of rank one.  Notice the
contrast: traditional Pontrjagin-Thom theory has \emph{stable} framings of
the \emph{normal} bundle, whereas $\tbordfr$~has \emph{unstable} framings of
the \emph{tangent} bundle.
 
By analogy to the homology homomorphism~\eqref{eq:1} we are led to the
following definition. 

  \begin{definition}[\cite{A2}]\label{thm:5}
 An \emph{$n$-dimensional topological field theory} is a homomorphism
  \begin{equation}\label{eq:2}
     F\:\tbord\longrightarrow (\Ab,\otimes ) 
  \end{equation}
of symmetric monoidal categories. 
  \end{definition}

\noindent
 As telegraphed in Remark~\ref{thm:2}, in a quantum field theory disjoint
unions map to \emph{tensor products}, not direct sums.  There are many
variations on this definition.  The domain can be a bordism category of
smooth manifolds with extra structure, or even of singular manifolds.  The
codomain may be replaced by any symmetric monoidal category, algebraic or
not.  We introduce a more drastic variant of Definition~\ref{thm:5}
in~\S\ref{sec:5}.  A typical choice for the codomain is $(\CVect,\otimes )$,
the category of complex vector spaces under tensor product.  A topological
field theory with values in~$\CVect$ is a linearization---a linear
representation---of manifolds.

We have been led naturally to Definition~\ref{thm:5} by combining basic ideas
in homology and bordism.  But this is hardly the historical path!  For that
we turn in the next section to notions in quantum field theory.  Before
leaving bordism, though, we pause to remind the reader of the connection with
Morse theory.

Intuitively, a Morse function refines the arrow of time to a particular time
function.  Let $X\:Y_0\to Y_1$~be an $n$-dimensional bordism.  A function
$f\:X\to\RR$ is compatible with the bordism structure if there exist
$t_0<t_1$ such that $t_0,t_1$~are regular values of~$f$, the image of~$f$ is
contained in~$[t_0,t_1]$, and $Y_i=f\inv (t_i)$.  Furthermore, $f$~is a
\emph{Morse function} if it has finitely many isolated \emph{nondegenerate}
critical points.  The main theorems in Morse theory~\cite{Mi2} assert that
slices~$f\inv (t)$ and~$f\inv (t')$ are diffeomorphic if there are no
critical values between~$t$ and~$t'$, and at an isolated critical point there
is a topology change which is described by a standard surgery.  For example,
\begin{figure}[h]
 \centering
 \includegraphics[scale=.8]{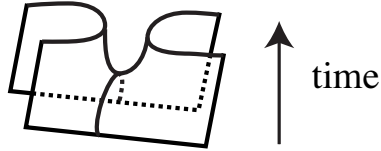}
 \caption{An elementary bordism}\label{fig:4}
 \end{figure}
 in Figure~\ref{fig:4} the local slice evolves from the two parallel line
segments at the bottom to the two curves at the top; the saddle depicts the
\emph{elementary bordism} which connects the two local slices.   
\begin{figure}[h]
 \centering
 \includegraphics[scale=.7]{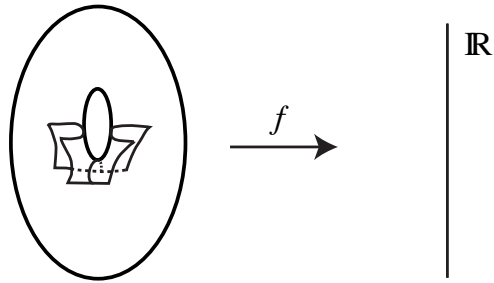}
 \caption{A Morse function}\label{fig:5}
 \end{figure}
Figure~\ref{fig:5} displays the standard example of a Morse function on the
torus---the height function---and embeds the elementary bordism of
Figure~\ref{fig:4} into a neighborhood of one of the critical points of
index~1.   

  \begin{remark}[]\label{thm:34}
 The local description of the topology change at a critical point uses
a manifold with \emph{corners}, as in Figure~\ref{fig:4}.  Manifolds with
boundary and no corners do not suffice.  The additional locality afforded by
admitting corners---and eventually higher codimensional corners---is a
crucial idea for the cobordism hypothesis; see~\S\ref{sec:5}. 
  \end{remark}

Morse functions exist, as a consequence of Sard's theorem.  This means that
any bordism can be decomposed as a composition of elementary bordisms, one
for each critical point.  Manipulations with Morse functions are a key
ingredient in Milnor's presentation~\cite{Mi3} of Smale's $h$-cobordism
theorem~\cite{Sm}.  The space of Morse functions on a fixed bordism has many
components: Morse functions in different components induce qualitatively
different decompositions into elementary bordisms.  Cerf~\cite{C} relaxed the
Morse condition to construct a \emph{connected} space of functions.  This
enables a systematic study of transitions between decompositions.  For
example, Cerf theory is the basis for Kirby calculus~\cite{K}, which
describes links in 3-manifolds and 4-manifolds.  As we shall see it is also a
crucial tool for constructing topological field theories.

An elementary illustrative example of a Cerf transition is the family of
functions
  \begin{equation}\label{eq:3}
     f_t(x) = \frac {x^3}3 - tx,\qquad x,t\in \RR. 
  \end{equation}
For~$t>0$ this is a Morse function with nondegenerate critical points
at~$x=\pm\sqrt{t}$.  For~$t<0$ it is a Morse function with no critical
points.  At~$t=0$ the function fails to be Morse: $x=0$~is a degenerate
critical point.  So as $t$~increases from negative to positive two critical
points are born on the $x$-line, and they separate at birth.  In the other
direction, as $t$~decreases from positive to negative the two critical points
collide and annihilate.  This simple ``birth-death transition'' is all that is
needed to connect different components of Morse functions.

   \section{Quantum field theory}\label{sec:3}

For much of its history quantum field theory was tied to four spacetime
dimensions and a handful of physically realistic examples.  As opposed to
quantum mechanics, where the underlying theory of Hilbert spaces and operator
theory has been fully developed, the analytic underpinnings of quantum field
theory remain unsettled.  Still, there has been a huge transformation over
the past three decades.  Quantum field theorists now study a large set of
examples in a variety of dimensions, not all of which are meant to be
physically relevant.  A deeper engagement with mathematicians and mathematics
has led physicists to study models whose consequences are more relevant to
geometry than to accelerators.  Topological and algebraic aspects of quantum
field theories have come to the fore.  From another direction string theory
has illuminated the subject, and there are new ties to condensed matter
theory as well.
 
In this section we briefly sketch how Definition~\ref{thm:5} of a topological
quantum field theory emerges from physics.  Our exposition is purely formal,
extracting the structural elements which most directly lead to our goal.
Let's begin with quantum mechanics, which is a 1-dimensional quantum field
theory.  (The dimension of a theory refers to spacetime, and at least in
mainstream theories there is a single time dimension.  Thus a 1-dimensional
theory only has time; space is treated externally.)  The basic ingredients
are a complex separable Hilbert space~$\sH$ and for each time interval of
length~$t$ a unitary operator
  \begin{equation}\label{eq:9}
     U_t = e^{-itH/\hbar} .
  \end{equation}
Here $H\:\sH\to\sH$~is the self-adjoint \emph{Hamiltonian} which describes
the quantum system, and $\hbar$~is Planck's constant.  The pure \emph{states}
of the system are vectors (really complex lines of vectors) in~$\sH$, and the
unitary operators~\eqref{eq:9} describe the evolution of a state in time.
Self-adjoint operators~$\obs$ on~$\sH$ act on the system---they are the
\emph{observables}---and the physics is encoded in \emph{expectation values}
  \begin{equation}\label{eq:4}
     \langle \Omega ,U_{t_n}\obs \mstrut _n\cdots U_{t_2}\obs \mstrut
     _2U_{t_1}\obs \mstrut _1U_{t_0}\Omega \rangle. 
  \end{equation}
In this expression the state~$\Omega $ evolves for time~$t_0$, is acted on by
the operator~$\obs_1$, then evolves for time~$t_1$, then is acted on by the
operator~$\obs_2$, etc.  See Figure~\ref{fig:5a} for a pictorial
representation.
  \begin{figure}[h]
  \centering
  \includegraphics[scale=.8]{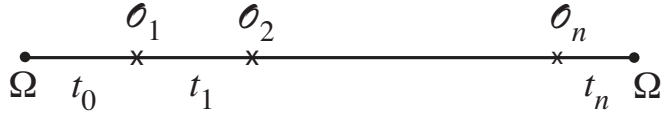}
  \caption{Vacuum expectation value in quantum mechanics}\label{fig:5a}
  \end{figure}
We recommend~\cite{Ma,Fa} for structural expositions of mechanics which
elucidate the pairing of states and observables.
 
It is convenient and powerful to analytically continue time~$t$ from the real
line to the complex line with the restriction $\Imm t<0$.  Real times are now
at the boundary of allowed complex times.  If the Hamiltonian~$H$ is
nonnegative, and $\Imm t<0$, then the evolution operator~$e^{-itH/\hbar}$ is
a contracting operator.  \emph{Wick rotation} to imaginary time is the
further restriction to purely imaginary~$t=\tau /\sqrt{-1}$, where the
\emph{Euclidean time}~$\tau $ is strictly positive.  We associate the
Euclidean contracting evolution~$F_\tau =e^{-\tau H/\hbar}$ to an interval of
length~$\tau $, that is, to a compact, connected Riemannian 1-manifold with
boundary whose total length is~$\tau $.  The evolution obeys a semigroup law
  \begin{equation}\label{eq:5}
     F_{\tau _2+\tau _1} = F_{\tau _2}\circ F_{\tau _1}, 
  \end{equation}
as illustrated in Figure~\ref{fig:6}.  This is already reminiscent of bordism.
  \begin{figure}[h]
  \centering
  \includegraphics[scale=.8]{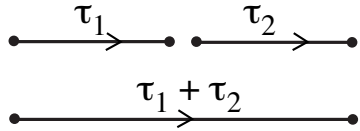}
  \caption{Composition of 1-dimensional bordisms}\label{fig:6}
  \end{figure}
We can imagine a bordism category $\Bord_{\langle 0,1 \rangle}^{\Riem}$ whose
objects are compact oriented 0-manifolds and whose morphisms are compact
Riemannian oriented 1-manifolds with boundary.  The semigroup law for the
evolution of a quantum mechanical system is encoded in the statement that
  \begin{equation}\label{eq:6}
     F\: \Bord_{\langle 0,1 \rangle}^{\Riem}\longrightarrow \Hilb 
  \end{equation}
is a homomorphism to the category of Hilbert spaces and contracting linear
maps.  Notice that $F$~ encodes more than evolution.  For example, we demand
that $F$~be a homomorphism of \emph{symmetric monoidal} categories mapping
disjoint unions to tensor products, which encodes the idea that the state
space of the union of quantum mechanical systems is a tensor product.  Exotic
``evolutions'' are now possible; see Figure~\ref{fig:6a}.  In a more careful
axiomatization~\cite{Se1} one takes the codomain to be a category of
topological vector spaces; then the Hilbert space structure emerges more
organically from the geometry, as do the operator insertions in~\eqref{eq:4}.
  \begin{figure}[h]
  \centering
  \includegraphics[scale=.8]{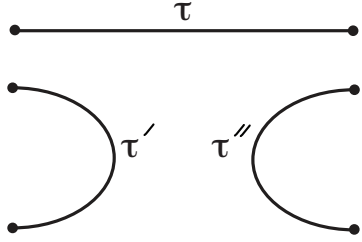}
  \caption{Exotic evolutions in quantum mechanics}\label{fig:6a}
  \end{figure}

It is a small step now to pass from the formal description~\eqref{eq:6} of a
quantum mechanical system to the assertion that an $n$-dimensional quantum
field theory is a homomorphism 
  \begin{equation}\label{eq:7}
     F\:\tbordr\longrightarrow \Hilb
  \end{equation}
from the bordism category of Riemannian $n$-dimensional bordisms
(``Riemannian spacetimes'') to the category of Hilbert spaces (better:
topological vector spaces).  If $X$~is such a bordism, and $x\in X$ a point
not on the boundary, then the boundary sphere of the geodesic ball of
sufficiently small radius~$r$ maps under~$F$ to a topological vector
space~$\sH_r$, and the limit as~$r\to0$ is a topological vector space of
operators associated to the point~$x$.  We can approximate it by the
topological vector space at some small finite radius~$r_0$.  Remove an open
ball of radius~$r_0$ about~$x$.  Choose the arrow of time so that the new
boundary component---the sphere of radius~$r_0$ about~$x$---is incoming.
  \begin{figure}[h]
  \centering
  \includegraphics[scale=.8]{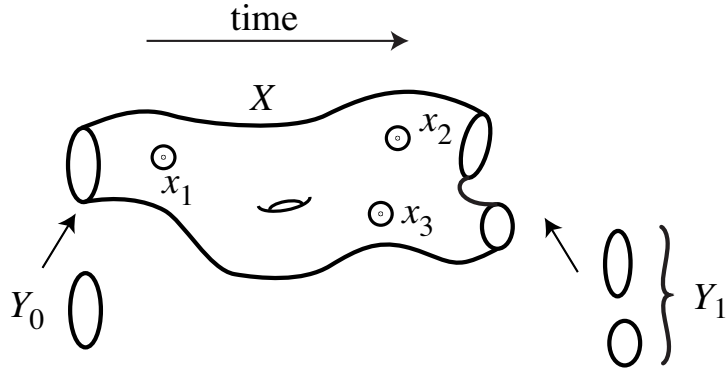}
  \caption{Operator insertions}\label{fig:7}
  \end{figure}
For example, the bordism in Figure~\ref{fig:7} has incoming boundary~$Y_0$
union the spheres about~$x_1$, $x_2$, and~$x_3$ and outgoing boundary~$Y_1$.
A field theory~$F$ determines topological vector spaces~$F(Y_0), F(Y_1)$ for
the boundary components and then topological vector spaces~$V_1,V_2,V_3$
associated to the points~$x_1,x_2,x_3$.  The bordism~$X$ goes over to a
linear map
  \begin{equation}\label{eq:8}
     F(X)\: V_1\otimes V_2\otimes V_3\longrightarrow \Hom\bigl(F(Y_0),F(Y_1)
     \bigr). 
  \end{equation}
This is the sense in which the topological vector spaces~$V_i$ attach a space
of operators to~$x_i$, analogously to the operators which appear
in~\eqref{eq:4} as illustrated in Figure~\ref{fig:5a}.  In case
$Y_0=Y_1=\emptyset ^{n-1}$, then $F(X)$~is called a \emph{correlation
function} between ``operators'' at the points~$x_i$.  If in addition there
are no points~$x_i$, then $F(X)$~is a complex number, the \emph{partition
function} of the closed manifold~$X$.
 
This geometric formulation of quantum field theory developed in the 1980s out
of the interaction between mathematicians and physicists centered around
2-dimensional conformal field theory.  Graeme Segal's samizdat manuscript
\emph{The definition of conformal field theory}, now published~\cite{Se2},
was widely distributed and very influential among both mathematicians and
physicists.  Segal's recent series of lectures~\cite{Se1} explores and
expands on these ideas in the context of general quantum field theories.
More traditional mathematical treatments of quantum field theory~\cite{SW},
\cite{H}, \cite{GJ} are set in four-dimensional Minkowski spacetime and focus
on analytic aspects.  The geometric formulation set the stage for the advent
of \emph{topological} field theories.  In~1988 Witten~\cite{W1} introduced
twistings of supersymmetric quantum field theories on Minkowski spacetime
which allow them to be formulated on arbitrary oriented Riemannian manifolds.
Special correlation functions in twisted theories are topological invariants.
Witten's first application was to a supersymmetric gauge theory in four
dimensions---a theory whose principal field is a connection on a principal
bundle---where he showed that Donaldson's polynomial invariants of
4-manifolds~\cite{D} are correlation functions in that twisted supersymmetric
gauge theory.  Two-dimensional supersymmetric $\sigma $-models---whose
principal field is a map $\Sigma \to M$ from a 2-manifold into a Riemannian
target manifold---also admit topological twistings in case there is enough
supersymmetry (which constrains the target manifold to be K\"ahler in the
basic case).  These 2-dimensional topological field theories~\cite{W2} have
had profound consequences for algebraic geometry in the form of Gromov-Witten
invariants and mirror symmetry.  By late~1988 Witten realized~\cite{W3} that
the Jones polynomials of knots and links in~$S^3$ are encoded in a
3-dimensional field theory---called \emph{Chern-Simons theory} after the
classical action functional of connections which defines it---and he used it
to introduce new invariants of 3-manifolds.  This theory, as opposed to the
topologically twisted supersymmetric models, is topological at the classical
level and has an immediate connection to combinatorially accessible
invariants.  For many mathematicians it served as an accessible entr\'ee into
quantum field theory.  In early~1989 Atiyah~\cite{A2} introduced a set of
axioms for topological quantum field theory which essentially amount to
Definition~\ref{thm:5}.

   \section{Topological quantum field theory}\label{sec:4}

In this section we flesh out Definition~\ref{thm:5} for simple 1-dimensional
and 2-dimensional theories.  The constructions and theorems give a taste of
what is possible in more complicated and interesting situations.  We include
a rigorous finite version of the Feynman path integral; the nonrigorous
infinite version is one of the main tools in a quantum field theorists'
arsenal.

  \subsection*{1-dimensional theories}\label{subsec:4.1}

Let us begin our exploration of Definition~\ref{thm:5} with a 1-dimensional
topological field theory of oriented manifolds.  Recall that the domain of
such a theory is the bordism category~$\Bord_{\langle 0,1 \rangle}^{SO}$ in
which an object is a compact oriented 0-manifold---a finite set of points,
each with a~`$+$' or~`$-$' attached---and a morphism is an oriented
1-dimensional bordism.  In more detail, if $X\:Y_0\to Y_1$ is an oriented
bordism, then $X$~is a manifold with boundary and so at $x\in \partial X$ we
have a short exact sequence of real vector spaces
  \begin{equation}\label{eq:40}
  0\longrightarrow T_x\partial X\longrightarrow T_xX\longrightarrow
  T_xX/T_x\partial X\longrightarrow 0 
  \end{equation}
The normal bundle carries a canonical orientation: vectors which exponentiate
to curves leaving the manifold are positively oriented.  However, when
interpreted as a bordism we use the arrow of time to orient the normal
bundle.  Namely, outgoing boundary components have the canonical orientation
and incoming boundary components the opposite to the canonical orientation.
Then using \eqref{eq:40}~ an orientation of~$X$ induces one on~$\partial X$,
and we require that the diffeomorphisms $Y_i\to\partial X_i$ in
Definition~\ref{thm:3} preserve the induced orientation.  There is a
time-reversal operation which reverses the arrow of time (swaps incoming and
outgoing), hence the orientation of the normal bundle at the boundary and so
too the induced boundary orientation.

There are two basic objects in~$\Bord_{\langle 0,1 \rangle}^{SO}$: the
$+$~point and the $-$~point.  Any other object is a tensor product (disjoint
union) of these.  Some basic morphisms are illustrated in Figure~\ref{fig:8}.
  \begin{figure}[h]
  \centering
  \includegraphics[scale=.8]{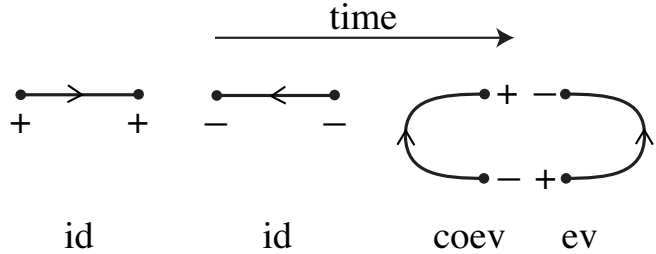}
  \caption{Elementary oriented 1-dimensional bordisms}\label{fig:8}
  \end{figure}
The arrow of time points to the right, whereas the orientation is notated by
an arrow on each component of the bordism.  Notice that there is a
correlation between the orientation, the arrow of time, and the boundary
orientation.  The first two morphisms are identities.  The third is called
\emph{coevaluation} and the fourth \emph{evaluation}.  The second bordism is
obtained from the first by time-reversal, and the same holds for the third
and fourth bordisms.  In this case time reversal is a \emph{duality}
operation: the $-$~point is the dual of the $+$~point and the evaluation is
dual to the coevaluation.  The coevaluation and evaluation are evolutions in
1-dimensional topological field theory which go beyond the standard
evolutions in quantum mechanics (Figure~\ref{fig:6a}).  Also, in quantum
mechanics the closed intervals are Riemannian, so have a length~$\tau $,
whereas in the topological theory all closed intervals are diffeomorphic and
lead to the identity evolution.  Comparison with~\eqref{eq:9} shows that the
Hamiltonian vanishes in a topological field theory.  There is no local
evolution: all of the non-identity behavior comes from topology.
 
Now suppose $F$~is a 1-dimensional oriented topological field
theory~\eqref{eq:2} with values in complex vector spaces:
  \begin{equation}\label{eq:10}
     F\: \bigl(\Bord_{\langle 0,1 \rangle}^{SO},\amalg \bigr)\longrightarrow
     \bigl(\CVect,\otimes \bigr).  
  \end{equation}
The notation recalls that $F$~is a homomorphism of symmetric monoidal
categories, so maps disjoint unions to tensor products.  The homomorphism~$F$
assigns a vector space $F(\pt_+)=V_+$ to the $+$~point and a vector space
$F(\pt_-)=V_-$ to the $-$~point.  This determines the value of~$F$ on all
compact oriented 0-manifolds as they are disjoint unions of $+$~and
$-$~points.  Also, since the empty 0-manifold~$\emptyset ^0$ is the tensor
unit for disjoint union, it maps under the homomorphism~$F$ to the tensor
unit for complex vector spaces under tensor product, which is the complex
line~$\CC$.  Next, consider $F$~evaluated on the bordisms in
Figure~\ref{fig:8}.  As $F$~is a homomorphism it sends identities to
identities, so the first two bordisms map to~$\id_{V_+}$ and~$\id_{V_-}$,
respectively.  The last two bordisms map under~$F$ to linear maps
  \begin{equation}\label{eq:12}
     \begin{matrix}
     &V_+\\c\:\CC\longrightarrow &\!\!\otimes \\&V_- 
      \end{matrix}\qquad \qquad \begin{matrix}
     \quad\;\, V_-\\e\:\otimes&\longrightarrow  &\CC\\
     \quad\;\, V_+
	\end{matrix}
  \end{equation}
where we have written the tensor product vertically to match the figure.
  \begin{figure}[h]
  \centering
  \includegraphics[scale=.8]{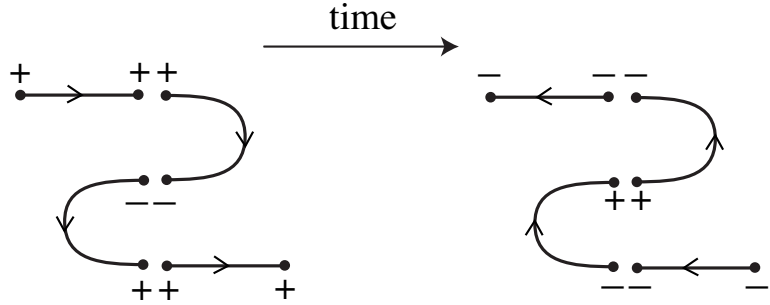}
  \caption{The S-diagrams}\label{fig:9}
  \end{figure}
The sense in which coevaluation and evaluation give rise to duality is
illustrated in Figure~\ref{fig:9}.  The left figure is the composition of two
1-dimensional bordisms, each with two components.  The first maps a single
$+$~point to the tensor product (disjoint union) of 3~points: $+$, $-$, $+$.
The second maps these 3~points back to the $+$~point.  The composition is
computed by gluing at the 3~points in the middle.  The result is
diffeomorphic to the identity map on the $+$~point.  Recall that morphisms in
$\Bord_{\langle 0,1 \rangle}^{SO}$ are 1-dimensional bordisms up to
diffeomorphisms which preserve the boundary identifications.  Comparing the
first composition in Figure~\ref{fig:9} with the first bordism in
Figure~\ref{fig:8} we see that the composition is the identity.  To see the
relation to duality we apply the homomorphism~$F$.  Now the homomorphism
property has two consequences: (1)\ $F$~sends a disjoint union of bordisms to
the tensor product of the corresponding linear maps, and (2)\ $F$~sends a
composition of bordisms to the corresponding composition of linear maps.
Using these rules we see that $F$~sends the compositions in
Figure~\ref{fig:9} to compositions of linear maps
  \begin{equation}\label{eq:11}
     \begin{matrix} V_+&\xrightarrow[\begin{matrix}\scriptstyle\otimes\\
     c\end{matrix}]{\quad \id_{V_+}\quad }&V_+&&\\[-24pt] &&\!\!\otimes&&\\
     &&V_-&&\\  
     &&\!\!\otimes&&\\[-33pt] &&V_+&\xrightarrow{\quad
     \begin{matrix}e\\\scriptstyle\otimes \\ 
     \id_{V_+}\end{matrix}\quad  }&V_+ \end{matrix} \qquad \qquad
     \begin{matrix} V_-&\xrightarrow[\begin{matrix}\scriptstyle\otimes\\
     c\end{matrix}]{\quad \id_{V_-}\quad }&V_-&&\\[-24pt] &&\!\!\otimes&&\\
     &&V_+&&\\ &&\!\!\otimes&&\\[-33pt] &&V_-&\xrightarrow{\quad
     \begin{matrix}e\\\scriptstyle\otimes \\ 
     \id_{V_+}\end{matrix}\quad  }&V_- \end{matrix} 
  \end{equation}
(Note that we have used the symmetry in the first diagram to exchange the
order of the tensor product in the maps~$c,e$ from~\eqref{eq:12}.)

  \begin{lemma}[]\label{thm:6}
 If the compositions~\eqref{eq:11} are identity maps, then $V_+,V_-$~are
finite dimensional vector spaces and $e$~is a nondegenerate duality pairing. 
  \end{lemma}

  \begin{proof}
 Set $c(1)=\sum\limits_{i=1}^Nv_+^i\otimes v_-^i$ for some $v_{\pm}^i\in
V_{\pm}$ and some positive integer~$N$.  Then the first composition
in~\eqref{eq:11} is the map $\xi \mapsto \sum e(v_-^i,\xi )v_+^i$.  Since
this is the identity map, it follows that $\{v_+^i\}_{i=1}^N$ spans~$V_+$,
whence $V_+$~is finite dimensional.  The same argument with the second
composition proves that $V_-$~is finite dimensional.  If $\xi \in V_+$
satisfies~$e(v_-,\xi )=0$ for all~$v_-\in V_-$, then $\xi =\sum e(v_-^i,\xi
)v_+^i=0$.  Similarly, using the second composition in~\eqref{eq:11} we
deduce that if $\eta \in V_-$ satisfies $e(\eta ,v_+)=0$ for all~$v_+\in
V_+$, then $\eta =0$.  Hence $e$~is a nondegenerate pairing.
  \end{proof}

  \begin{remark}[]\label{thm:8}
 A similar argument for a field theory $F\: \bigl(\Bord_{\langle 0,1
\rangle}^{SO},\amalg \bigr)\longrightarrow \bigl(\Ab,\otimes \bigr)$ with
values in abelian groups proves that $F(\pt_+)$~is finitely generated and
free. 
  \end{remark}

Lemma~\ref{thm:6} illustrates an important \emph{finiteness} principle in
topological field theories: the vector space attached to an $(n-1)$-manifold
in an $n$-dimensional topological field theory with values in~$\CVect$ is
finite dimensional.  We derived this finiteness from duality: the $+$~point
and $-$~point are duals, and that duality is expressed as the existence of
coevaluation and evaluation maps.  Notice that any vector space~$V$ has a
dual space, defined algebraically as the space of linear maps $V\to \CC$,
which comes with a canonical evaluation map.  However, the coevaluation map
exists if and only if $V$~is finite dimensional.  

This notion of finiteness generalizes to any symmetric monoidal category.

  \begin{definition}[]\label{thm:7}
 Let $\sC$ be a symmetric monoidal category and $x\in \sC$.  Then
\emph{duality data} for~$x$ is a triple~$(x',c,e)$ consisting of an
object~$x'\in \sC$, a coevaluation $c\:1\to x\otimes x'$, and an
evaluation~$e\:x'\otimes x\to 1$ such that the compositions 
  \begin{equation}\label{eq:13}
     x\xrightarrow{c\otimes \id_x}x\otimes x'\otimes x\xrightarrow{\id_x\otimes
     e}x \qquad \qquad  
     x'\xrightarrow{\id_{x'}\otimes c}x'\otimes x\otimes x'\xrightarrow{e\otimes
     \id_{x'}}x'
  \end{equation}
are identity maps.  We say \emph{$x$ is dualizable} if there exists duality
data for~$x$.
  \end{definition}

\noindent
 The argument in Lemma~\ref{thm:6} with the S-diagrams in Figure~\ref{fig:9}
applies in any $n$-dimensional field theory---take the Cartesian product of the
S-diagrams with a fixed $(n-1)$-manifold---which shows that objects in the
image of a field theory~$F$ are always dualizable.  In the next section we
define an extension of the notion of a field theory and there is a
corresponding extension of dualizability, which we take up in~\S\ref{sec:6}.
 
At this point we can state and prove a very simple special case of the
cobordism hypothesis. 

  \begin{theorem}[]\label{thm:9}
 Let $V$~be a finite dimensional complex vector space.  Then there is a
homomorphism~$F$ as in~\eqref{eq:10} such that $F(\pp)=V$. 
  \end{theorem}

  \begin{proof}
 If $Y$~is an oriented compact 0-manifold set 
  \begin{equation}\label{eq:14}
     F(Y)= \bigotimes \limits_{y\in Y:y=\pp}\!\!\!\!\! V \quad \otimes\;
     \bigotimes\limits_{y\in Y:y=\pt_-}\!\!\!\!\! V^*. 
  \end{equation}
Referring to the third and fourth bordisms in Figure~\ref{fig:8} define
$F(\coev)$ as the map $\CC\to V\otimes V^*$ which takes~$1\in \CC$ to the
identity map~$\id_V\in \End(V)\cong V^*\otimes V$ and $F(\ev)$ as the duality
pairing $V^*\otimes V\to \CC$.  A Morse function on a 1-dimensional bordism
decomposes it as a composition of the elementary bordisms coev and ev: a
nondegenerate critical point of a real-valued function on a 1-manifold is
either a local maximum or a local minimum.  The only Cerf move
(Figure~\ref{fig:11}) cancels a local maximum against a local minimum, and
the proof that this does not change the value of~$F$ is the statement that
the S-diagrams in Figure~\ref{fig:9} map to the identity.
  \begin{figure}[h]
  \centering
  \includegraphics[scale=.8]{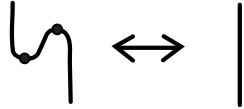}
  \caption{Cerf move in dimension one}\label{fig:11}
  \end{figure}
  \end{proof}

  \subsection*{2-dimensional theories}\label{subsec:4.2}

Next, consider a 2-dimensional oriented topological field theory 
  \begin{equation}\label{eq:15}
     F\: \bigl(\Bord_{\langle 1,2 \rangle}^{SO},\amalg \bigr)\longrightarrow
     \bigl(\CVect,\otimes \bigr). 
  \end{equation}
There is only one compact connected oriented 1-manifold up to diffeomorphism:
a circle has orientation-reversing diffeomorphisms (reflection).   
  \begin{figure}[h]
  \centering
  \includegraphics[scale=.8]{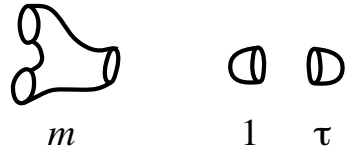}
  \caption{Some elementary oriented 2-dimensional bordisms}\label{fig:12}
  \end{figure}
Let $V=F(\cir)$.  Elementary 2-dimensional bordisms, as depicted in
Figure~\ref{fig:12}, give extra structure on~$V$, namely linear maps
  \begin{equation}\label{eq:16}
 \xymatrix@R-21pt{\txt<-12pc>{$m\:$}&V\otimes
V\ar[r]&V\\\txt<-12pc>{$1\:$}&\CC\ar[r]&V\\\txt<-12pc>{$\tau\:$}&V\ar[r]&\CC} 
  \end{equation}
The multiplication~$m$ gives~$V$ an algebra structure with respect to which
the image of~$1\in \CC$ under the linear map~$1$ is an identity element.  The
linear map~$\tau $ is a trace on~$V$.\footnote{Note that the bordism~ $\tau $
is the time-reversal of~$1$.  There is also a time-reversal of~$m$, which may
be expressed as a composition of the maps in~\eqref{eq:16} together with the
inverse to the nondegenerate bilinear pairing~$\tau \circ m$.}  Standard
arguments with oriented surfaces and their diffeomorphisms prove that $m$~is
associative and commutative and that the trace is nondegenerate in the sense
that the pairing $v_1,v_2\mapsto (\tau \circ m)(v_1,v_2)$ is a nondegenerate
pairing on~$V$.  For example, the composition of the bordisms labeled~$m$
and~$\tau $ in Figure~\ref{fig:12} is the product of the circle with the
bordism labeled~$\ev$ in Figure~\ref{fig:8}; then the argument of
Lemma~\ref{thm:6} with the S-diagram proves that the pairing $\tau \circ m$
is nondegenerate.  Thus an oriented 2-dimensional topological field theory
determines a \emph{commutative Frobenius algebra}, a commutative algebra with
a nondegenerate trace.  The converse is also true.

  \begin{theorem}[]\label{thm:10}
 Let $V$~be a commutative Frobenius algebra.  Then there is a homomorphism 
  \begin{equation}\label{eq:17}
     F\: \bigl(\Bord_{\langle 1,2 \rangle}^{SO},\amalg \bigr)\longrightarrow
     \bigl(\CVect,\otimes \bigr) 
  \end{equation}
with $F(\cir)=V$. 
  \end{theorem}

\noindent
 This is one of the oldest theorems in the subject.  In the physics
literature the statement dates at least to Dijkgraaf's thesis~\cite{Di}.
There are several proofs in the mathematics literature, for example in
\cite{Ab,Ko}.  The appendix to~\cite{MS} contains a proof of
Theorem~\ref{thm:10} as well as several important variations.  As in the
proof of Theorem~\ref{thm:9} we first extend~$F$ to all closed oriented
1-manifolds via tensor products.  The data~\eqref{eq:16} which defines the
Frobenius structure on~$V$ tells what to attach to elementary 2-dimensional
bordisms arising from critical points of a Morse function of index 1,0,2.  It
remains to verify that different Morse functions lead to the same linear map.
That check, for which we refer to the reader to~\cite{MS}, uses the basic
properties of a commutative Frobenius algebra.
 
These explicit arguments with Morse functions quickly become tedious and
difficult to execute.  The situation simplifies for \emph{extended} field
theories~(\S\ref{sec:5}) which are more local.  They are the province of the
cobordism hypothesis.  The cobordism hypothesis is proved using on the one
hand more powerful results about spaces of Morse functions and on the other
more sophisticated algebra to organize the argument.
 
One example of a commutative Frobenius algebra is the cohomology algebra
$H^{\bullet }(M;\CC)$ of a compact oriented $n$-manifold~$M$.  The trace is
pairing with the fundamental class ~$[M]\in H_{n}(M)$.  If there is odd
cohomology, then it is commutative in the graded sense because of signs in
the commutation rule for cup products.  For example, if $M=S^2$ then we
obtain the truncated polynomial algebra $\CC[x]/(x^2)$.  The corresponding
2-dimensional topological field theory plays a role in the construction of
Khovanov homology for links~\cite{Kh,B-N}.  If the Frobenius algebra $V$~is
\emph{semisimple}, then we can simultaneously diagonalize the multiplication
operators $M_a(b)=ab,\;a,b\in V$ and so find a basis of commuting idempotents
$e_1,e_2,\dots ,e_n\in V$: thus $e_ie_i=e_i$ and $e_ie_j=0$~if $i\not= j$.
The Frobenius algebra is determined up to isomorphism by nonzero complex
numbers~$\lambda _1,\lambda _2,\dots ,\lambda _n$ defined by~$\tau
(e_i)=\lambda _i$.  In this case everything in the field theory~$F$ with
$F(\cir)=V$ is easily computed in terms of the basis~$\{e_i\}$ and the
numbers~$\lambda _i$.  For example the 2-holed torus in Figure~\ref{fig:16}
maps to the endomorphism $e_i\mapsto \lambda _i\inv e_i$ of~$V$ and a closed
surface~$X_g$ of genus~$g$ maps to the complex number
  \begin{equation}\label{eq:18}
     F(X_g)=\sum \lambda _i^{1-g} .
  \end{equation}
These computations are made by chopping the surfaces into the elementary
bordisms in Figure~\ref{fig:12} and their time-reversals.
  \begin{figure}[h]
  \centering
  \includegraphics[scale=.8]{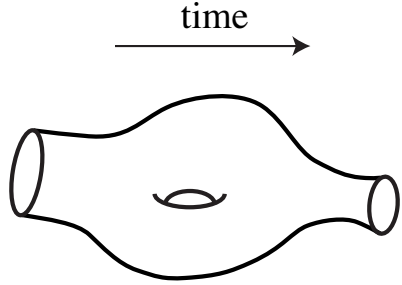}
  \caption{Torus with incoming and outgoing boundary circles}\label{fig:16}
  \end{figure}

Let $G$~be a finite group and $A=\Map(G,\CC)$ the vector space of
complex-valued functions on~$G$.  Then $A$~is an associative algebra under
the convolution product 
  \begin{equation}\label{eq:19}
     (f_1*f_2)(g) = \sum\limits_{g_1g_2=g}f_1(g_1)f_2(g_2),\qquad
     g,g_1,g_2\in G,\quad f_1,f_2\:G\to\CC. 
  \end{equation}
We also define the trace 
  \begin{equation}\label{eq:20}
     \tau (f)=\frac{f(e)}{\#G} ,
  \end{equation}
where $e\in G$ is the identity element.  The product is not commutative if
$G$~is not abelian.  Let $V$~be the center of~$A$, the space of \emph{class
functions} on~$G$; it is a commutative Frobenius algebra which can be
identified with the complexification $R(G)\otimes \CC$ of the representation
ring of~$G$.  Let $F_G$~denote the 2-dimensional oriented topological field
theory with $F_G(\cir)=V$ guaranteed by Theorem~\ref{thm:10}.  The
complexified representation ring is semisimple.  Classical orthogonality
formulas of Schur show that the characters~$\chi _i$ of the irreducible
complex representations of~$G$ are, up to scale, the commuting idempotents
$e_i=\bigl(\chi _i(1)/\#G \bigr)\,\chi _i$.  Then we easily compute that
$\lambda _i=\sum \chi _i(1)^2/\#G$ and from~\eqref{eq:18} the partition
function of a closed connected oriented surface is
  \begin{equation}\label{eq:21}
     F_G(X) = \sum\limits_{\substack{\chi \textnormal{
     irreducible}\\\textnormal{character of $G$}}}
     \left(\frac{\chi (e)}{\#G}\right) ^{\Euler(X)}, 
  \end{equation}
where $\Euler(X)$~is the Euler characteristic of~$X$.

The construction of~$F_G$ which relies on Theorem~\ref{thm:10} takes as input
the complexified representation ring and uses Morse theory to produce a
topological field theory.  There is also a direct geometric construction of
this simple finite theory.  For any manifold~$M$ let $\fld M$ denote the
collection of principal $G$-bundles $P\to M$.  So $P$~is a manifold with a
free right $G$-action and quotient~$M$.  In other terms $P\to M$ is a
covering space which is regular (Galois), but note that $P$~need not be
connected.  For example, if $M=\cir$ and $G=\zmod n$ for some positive
integer~$n$, then there are $n$~distinct isomorphism classes of principal
$G$-bundles over~$M$; the connectivity of the total space of a cover depends
on the prime factorization of~$n$.  For any manifold $\fld M$ is a category:
a morphism $(P'\to M)\longrightarrow (P\to M)$ is a smooth map $\varphi
\:P'\to P$ which commutes with the $G$-action and covers the identity map
of~$M$.  This category is a \emph{groupoid} since all morphisms are
invertible.  For $M=\pt$ there is only one $G$-bundle up to isomorphism, the
trivial bundle $P=G$ with $G$~acting by right multiplication, and the group
of automorphisms is $G$~acting by left multiplication on~$P$.
Figure~\ref{fig:13} depicts a groupoid equivalent to~$\fld{\pt}$.
  \begin{figure}[h]
  \centering
  \includegraphics[scale=.8]{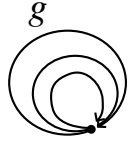}
  \caption{$G$-bundles over $\pt$}\label{fig:13}
  \end{figure}
There is a single object, the set of arrows is~$G$, and composition of arrows
is given by the group law.  For $M=\cir$ if we introduce a basepoint $p\in P$
on a $G$-bundle $P\to\cir$, then we can compute the holonomy, or monodromy,
around the circle (after choosing an orientation), which is an element
of~$G$.  The bundle with basepoint is rigid: any automorphism which fixes the
basepoint is the identity.  The group~$G$ acts simply transitively on the set
of basepoints over a fixed point of~$\cir$, and it conjugates the holonomy.
In this way we see that $\fld{\cir}$~is equivalent to the groupoid~$G\gpd G$
of $G$~acting on itself by conjugation.  It is depicted in
Figure~\ref{fig:14}.  The set of isomorphism classes~$\pi _0(\fld{\cir})$ is
the set of conjugacy classes in~$G$
  \begin{figure}[h]
  \centering
  \includegraphics[scale=.8,trim=100 0 0 0]{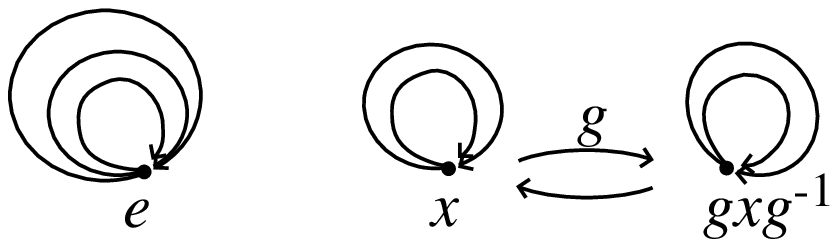}
  \caption{$G$-bundles over~$\cir$}\label{fig:14}
  \end{figure}
and the automorphism group $\pi _1(\fld{\cir},P)$ at a $G$-bundle with
holonomy~$x$ is the centralizer group of~$x$ in~$G$. 
 
Principal $G$-bundles are local and contravariant.  Consider a bordism, as in
Figure~\ref{fig:1} with the arrow of time pointing to the right.  The
inclusions of the incoming and outgoing boundary induce restriction maps of
bundles, which are homomorphisms of groupoids:
  \begin{equation}\label{eq:23}
     \xymatrix@!C{&\fld{X}\ar[dl]_s \ar[dr]^t \\ \fld{Y_0} &&
     \fld{Y_1}} 
  \end{equation}
A diagram of the form~\eqref{eq:23} is a \emph{correspondence}, which is a
generalization of a homomorphism from~$\fld{Y_0}$ to~$\fld{Y_1}$.  Namely, if
$s$~is invertible, then $s\times t$~embeds $\fld X$ into~$\fld{Y_0}\times
\fld{Y_1}$ as the graph of~$t\circ s\inv $.  A composition of bordisms
(Figure~\ref{fig:3}) induces a composition of correspondences
  \begin{equation}\label{eq:24}
     \xymatrix@!C{&&\fld{X''\circ X'}\ar[dl]_{r'} \ar[dr]^{r''} \\
     &\fld{X'}\ar[dl]_{s'} \ar[dr]^{t'}&&\fld{X''}\ar[dl]_{s''} \ar[dr]^{t''} \\
     \fld{Y_0} && \fld{Y_1} && \fld{Y_2}} 
  \end{equation}
The locality of principal $G$-bundles is hidden in this statement: the
groupoid~$\fld{X''\circ X'}$ of $G$-bundles on the composition~$X''\circ X'$
is the fiber product of~$t'$ and~$s''$; that is, a $G$-bundle $P\to X''\circ
X'$ is a triple $(P',P'',\theta )$ consisting of $G$-bundles $P'\to X'$,
$P''\to X''$, and an isomorphism $\theta \:P'\res{Y_1}\to P''\res{Y_1}$ of
their restrictions to~$Y_1$.
 
Correspondence diagrams can often be ``linearized'' into honest maps.  For
the field theory~$F_G$ we use closed oriented 1-manifolds~$Y$ and compact
oriented 2-dimensional bordisms~$X$.  On 1-manifolds we define 
  \begin{equation}\label{eq:25}
     F_G(Y) = \Hom\bigl(\fld Y,\CC \bigr). 
  \end{equation}
Here we view~$\CC$ as a groupoid with only identity morphisms.  Then
homomorphisms $\fld Y\to\CC$ assign complex numbers to objects in~$\fld Y$ so
that the numbers at each end of a morphism are equal.  In other words,
$\Hom(\fld Y,\CC)$~is the vector space of invariant functions on~$\fld Y$, so
can be identified with $\Map\bigl(\pi _0(\fld Y),\CC \bigr)$, the space of
functions on equivalence classes of $G$-bundles.  Then to a
correspondence~\eqref{eq:23} we define
  \begin{equation}\label{eq:26}
     F_G(X) = t_*\circ s^*\:F_G(Y_0)\longrightarrow F_G(Y_1) 
  \end{equation}
as pullback followed by pushforward.  The fibers of~$t$ are (equivalent to)
groupoids with finitely many objects, each with a finite stabilizer group.
The pushforward~$t_*$ of a function~$\phi $ on~$\fld X$ is the sum
  \begin{equation}\label{eq:27}
     t_*(\phi )(y) = \sum\limits_{x}\frac{\phi (x)}{\#\Aut(x)},\qquad y\in
     \fld{Y_1}, 
  \end{equation}
over the equivalence classes~$x$ in the fiber~$t\inv (y)$ of the value
of~$\phi $ divided by the order of the automorphism group.  (This formula
makes clear that $F_G$~may be defined on rational vector spaces.)  Key point:
the fact that \eqref{eq:24}~is a fiber product implies that the push-pull
construction takes compositions of bordisms to compositions of linear maps.
In other words, there is an \emph{a priori} proof that the push-pull
construction produces a homomorphism~$F_G\:\Bord_{\langle 1,2
\rangle}^{SO}\to\CVect$ of symmetric monoidal categories.  The enterprising
reader can now \emph{compute} that $F_G(\cir)$~is the vector space of
central functions on~$G$, and that the basic bordisms in Figure~\ref{fig:12}
map to the convolution product, the character of the identity representation,
and the trace~\eqref{eq:20}. 

Now suppose $X$~is a closed oriented 2-manifold.  It is interpreted as a
bordism $X\:\emptyset ^1\to\emptyset ^1$.  In grand Bourbaki style the
groupoid of $G$-bundles~$\fld{\emptyset ^1}$ has a single object with only
the identity morphism.  (After all, $\mathscr{F}$~maps disjoint unions to
Cartesian products, and $\emptyset ^1$~is the tensor unit for disjoint
union.)  In this case \eqref{eq:26}~specializes to the sum of the constant
function~1 over~$\fld X$: it counts (with automorphisms) the $G$-bundles
over~$X$.  If $X$~is connected then that count of bundles is
  \begin{equation}\label{eq:28}
     F_G(X) = \frac{\#\Hom\bigl(\pi _1(X,x),G\bigr)}{\#G};
  \end{equation}
the numerator counts $G$-bundles with a basepoint over~$x$ and the group~$G$
acts simply transitively on the basepoints.   

  \begin{theorem}[]\label{thm:12}
 Let $X$~be a compact oriented connected 2-manifold and $G$~a finite group.
Then 
  \begin{equation}\label{eq:29}
     \#\Hom\bigl(\pi _1(X,x),G\bigr) = (\#G)\sum\limits_{\substack{\chi
     \textnormal{ irreducible}\\\textnormal{character of $G$}}}
     \left(\frac{\chi (1)}{\#G}\right) ^{\Euler(X)},
  \end{equation} 
where $\Euler(X)$~is the Euler characteristic of~$X$.
  \end{theorem}

\noindent
 The theorem follows immediately by comparing~\eqref{eq:28}
and~\eqref{eq:21}.  It was known to Frobenius and Schur from the character
theory of finite groups, with no quantum fields in sight.  The proof given
here is representative of how topological field theory is used in more
complicated situations.  The invariant on the left hand side
of~\eqref{eq:29}, initially defined for closed 2-manifolds, is extended to an
invariant for compact 2-manifolds with boundary which obeys a gluing law.  So
it is computed by chopping~$X$ into elementary pieces (as in
Figure~\ref{fig:12} together with the time-reversal of~$m$).

  \begin{remark}[]\label{thm:30}
 The appearance of the Euler characteristic in~\eqref{eq:29} suggests an
extension of~$F_G$ which includes 0-manifolds.  They would appear as corners
of 2-manifolds and boundaries of 1-manifolds.  Then in a triangulation
of~$X$, the count of vertices, edges, and triangles in the triangulation
should combine to give the Euler characteristic~$\Euler(X)$ and a new proof
of \eqref{eq:29}.  In such an \emph{extended} field theory we have more
locality, so more decompositions and hence more computational flexibility.
We take up extended theories in~\S\ref{sec:5} and pursue this idea in
Example~\ref{thm:18}.
  \end{remark}

  \begin{remark}[]\label{thm:11}
 There is a variation on~\eqref{eq:26} in which $\fld X$ in~\eqref{eq:23}
carries an integral kernel.  In that case the pull-push formula~\eqref{eq:26}
is modified to pull-multiply-push.  The integral kernel must be local in that
it multiplies in the fiber product~\eqref{eq:24}.  In this 2-dimensional
theory we can obtain such an integral kernel by starting with a cocycle for a
class in the group cohomology~$H^2(G;\CC/\ZZ)$.
  \end{remark}

The theory~$F_G$ was introduced by Dijkgraaf and Witten~\cite{DW}.
See~\cite{FQ},\cite{F} for more details about defining~$F_G$ by counting
principal $G$-bundles.  The lecture notes~\cite{Q} contain elaborations and
many more examples.

The push-pull construction is a finite version of the Feynman functional
integral in quantum field theory.  The groupoid~$\fld M$ consists of
\emph{gauge fields} for a finite group~$G$; if $G$~is a Lie group, then gauge
fields form the groupoid of $G$-connections on~$M$.  The integral kernel
described in Remark~\ref{thm:11} is the exponential of the \emph{classical
action} of the field theory.  The pushforward~$t_*$ is the \emph{Feynman
integral} or \emph{functional integral} or \emph{path integral} over the
space of fields (with fixed boundary condition).  In almost all physically
interesting examples the space, or stack, of fields is not finite, but rather
is infinite dimensional.  One way to define pushforward~$t_*$ on functions is
via integration theory, which of course requires a measure on the space of
fields.  (There are alternatives, at least for some topological theories;
see~\cite{FHT} for one example.)  Furthermore, the measures must be
consistent with the fiber product~\eqref{eq:24} under composition of
bordisms.  Such measures have not been constructed rigorously in most
examples of physical interest.  The example of finite gauge theories, while
it nicely illustrates many topological and algebraic aspects, misses
completely the central analytical issues in quantum field theory.

   \section{$n$-categories and extended topological quantum field
theory}\label{sec:5} 

In this section we extend the definition of an $n$-dimensional topological
field theories in two directions: (i)\ to invariants of manifolds of all
dimensions~$\le n$ and (ii)\ to invariants of families of manifolds.  These
extensions go beyond what was traditionally done in quantum field theory.  

Standard topological field theories, as in Definition~\ref{thm:5}, are
\emph{local} in that invariants of $n$-manifolds are computed by cutting
along closed codimension~1 submanifolds.  We saw after Theorem~\ref{thm:12},
and even in the description of classical Morse theory (Remark~\ref{thm:34}),
that it is desirable to go further and cut along codimension~2 submanifolds
as well, so have $n$-manifolds with corners.  Once we take that plunge we may
as well continue cutting in higher and higher codimension until we are
cutting along 0-manifolds.  In other words, we end up considering
$n$-manifolds with corners of all codimension.  The local model for the
maximal corner is a corner in real affine space: $\{(x^1,x^2,\dots ,x^n)\in
\AA^n: x^i\ge0\}$ near~$(0,0,\dots ,0)$.
 
In a bottom up view, rather than a top down view, we build higher dimensional
manifolds by time evolution of lower dimensional manifolds.  This is
illustrated in Figure~\ref{fig:8} by the time evolution of 0-manifolds to
produce 1-manifolds.  Now we evolve again, introducing a second time as in
Figure~\ref{fig:15}.
  \begin{figure}[h]
  \centering
  \includegraphics[scale=.8]{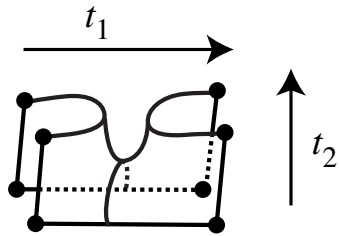}
  \caption{Two-time evolution of two points}\label{fig:15}
  \end{figure}
Let $t_1,t_2\in [0,1]$ denote the times, so the space of times is the
square~$[0,1]\times [0,1]$.  At each of the four corners $t_1,t_2\in \{0,1\}$
lies the 0-manifold~$Y$ consisting of two points.  At time~$t_2=0$ they
evolve in~$t_1$ via the identity bordism, whereas at time~$t_2=1$ they evolve
as the evaluation followed by the coevaluation.  (These 1-dimensional
bordisms are pictured in Figure~\ref{fig:8}).  The evolution in~$t_2$ is a
2-dimensional bordism~$W$ between these two 1-dimensional bordisms~$X_0,X_1$.
As a manifold it is a 2-dimensional manifold with corners, but as a bordism
we remember the time evolutions.  Morally, as in~\S\ref{sec:2} it is only the
arrows of time which matter---and these only near the boundaries and
corners---but it is convenient both heuristically and technically to think in
terms of actual time functions.  An algebraic representation of this two-time
evolution is:
  \begin{equation}\label{eq:30}
     \xymatrix@C+24pt{Y\rtwocell<7>^{X_1 }_{X_0}{^{\qquad\;\;\; W}}& Y}
  \end{equation}
 
The algebraic structure which includes~\eqref{eq:30} is a \emph{2-category}.
In addition to objects~$x,y$ and morphisms $f,g\:x\to y$ mapping between
them, there are now 2-morphisms $\eta \:f\Rightarrow g$ which map between
morphisms.  For clarity `morphisms' are now termed `1-morphisms'.  In the
2-category~$\bordt$ the objects are compact 0-manifolds, the
1-morphisms are 1-dimensional bordisms, and the 2-morphisms are 2-dimensional
bordisms.  A 2-category has two associative composition laws, easily seen
pictorially in~$\bordt$.  Namely, we can compose horizontally in the first
time~$t_1$ or vertically in the second time~$t_2$.  Disjoint union is an
extra algebraic structure---still called a symmetric monoidal structure---and
the empty manifolds are identity elements for disjoint union.  So, for
example, a closed 2-manifold~$W$ is interpreted as a 2-morphism $W\:\emptyset
^1\Rightarrow\emptyset ^1$ in~$\bordt$.  For now we leave unspecified what
sort of extra topological data (orientation, framing, \dots ) we assume
present.

The saddle in Figure~\ref{fig:15} is the elementary bordism in Morse theory
depicted in Figure~\ref{fig:4}.  In other words, it is the
2-manifold~$D^1\times D^1$ which implements the surgery beginning with
$S^0\times D^1$ and ending with $D^1\times S^0$.  Here $D^1$~is the standard
closed 1-ball.  The general surgery 
  \begin{equation}\label{eq:31}
     D^p\times D^q\:S^{p-1}\times D^q\longrightarrow D^p\times S^{q-1}, 
  \end{equation}
can be written algebraically in a diagram similar to~\eqref{eq:30}
with~$Y=S^{p-1}\times S^{q-1}$.  Morse theory tells that a manifold has a
handlebody decomposition into elementary bordisms~\eqref{eq:31}.  We might
conclude that 2-categories go far enough, and that nothing is to be gained by
chopping further.  We could, after all, make a 2-category whose objects are
closed $(n-2)$-manifolds and with 1-morphisms and 2-morphisms their time
evolutions.  But the structure simplifies if we don't stop there and rather
go all the way down to points.
 
Therefore, to study manifolds of dimension~$\le n$, or equivalently to study
topological field theories of dimension~$n$, we are led to the
\emph{$n$-category}~$\bordn$ whose objects are compact 0-manifolds and whose
$k$-morphisms ($1\le k\le n$) are $k$-time evolutions of objects.  There are
$k$~composition laws for $k$-morphisms, and they satisfy various
compatibilities.  Disjoint union gives a symmetric monoidal structure.  It is
a complicated combinatorial problem to track all of this data.  The relevance
of higher categories to topological field theory was understood in the early
1990s, but at that time rigorous foundations were not available.  In the
intervening years several approaches and definitions have been advanced.  We
will not attempt a formal definition here, but refer the reader
to~\cite{BD,L1} for more detailed exposition and references.
 
The $n$-category $\bordn$ is the first extension we envisioned at the
beginning of this section.  The second is to families of manifolds.  It turns
out that this can be encoded by extending the $n$-category~$\bordn$ higher
up: we adjoin $(n+1)$-morphisms, $(n+2)$-morphisms, etc.  Namely, if
$W_0,W_1$~are $n$-dimensional bordisms we define an $(n+1)$-morphism $\varphi
\:W_0\to W_1$ to be a diffeomorphism which preserves all of the ``boundary
data''.  An $n$-morphism is a map between two $(n-1)$-morphisms, each of
which is a map between two $(n-2)$-morphisms, and on down.  The
diffeomorphism~$\varphi $ must preserve the implicit identifications.  In
terms of the $n$-time evolution, $\varphi $~must be compatible with the data
at each of the $2^n$~extreme times~$t_i\in \{0,1\}$.  Since $\varphi $~is a
diffeomorphism, it is invertible.  We continue and define an $(n+2)$-morphism
$\varphi _0\to\varphi _1$ to be an isotopy between the
diffeomorphisms~$\varphi _0$ and~$\varphi _1$, again preserving the boundary
data.  Isotopies are also invertible, up to a higher isotopy.  Continuing in
this way we have $k$-morphisms for all~$k$, so an $\infty $-category.  But it
has the property that every $k$-morphism for~$k>n$ is invertible.

  \begin{definition}[]\label{thm:15}
 Let $n\in \ZZ^{>0}$.  An $(\infty ,n)$-category is an $\infty $-category in
which every $k$-morphism is invertible for~$k>n$. 
  \end{definition}

\noindent
 `Definition' is not really appropriate as we have not defined $\infty
$-categories!  There are complete definitions for $\infn$-categories, in fact
several \cite{Ba,R,Be} with others on the way, and also a study~\cite{BS} of
all homotopy theories of $\infn$-categories.

  \begin{remark}[]\label{thm:19}
 A higher category in which \emph{every} morphism is invertible---i.e., an
$(\infty ,0)$-category---is a combinatorial model for a space.  Since every
morphism is invertible, this is also called an \emph{$\infty $-groupoid.}  So
whereas an $n$-category has \emph{sets} of $n$-morphisms, an $\infn$-category
has \emph{spaces} of $n$-morphisms.  An $n$-category may be extended to a
\emph{discrete} $\infn$-category in which all $k$-morphisms for~$k>n$ are
identity maps.
  \end{remark}

  \begin{definition}[]\label{thm:16}
 $\bord n$ is the $\infn$-category whose objects are compact 0-manifolds,
$k$-morphisms for $1\le k\le n$ are $k$-time evolutions of objects, and
$k$-morphisms for~$k>n$ are $(k-n)$-fold iterated isotopies of
diffeomorphisms.  It is symmetric monoidal under disjoint union.
  \end{definition}

\noindent
 Again this is only a descriptive definition. 
 
The manifolds in $\bord n$ typically carry extra data.  For example, there is
an $\infn$-category $\bordso n$ of oriented bordisms.  There is also a
bordism category of bordisms with tangential \emph{framing}, but in an
unstable\footnote{Framings on manifolds used to define framed bordism
groups---isomorphic by the Pontrjagin-Thom construction to stable homotopy
groups of spheres---are stable framings of the normal bundle.}  sense.
Namely, an \emph{$n$-framing} on a $k$-bordism~$W$ in~$\bordfr n$ is a
trivialization of~$TW\oplus (n-k)$, where $(n-k)$~is the trivial bundle of
the indicated rank.  The $\infn$-category of unoriented manifolds is denoted
$\bordo n$.  We use~ `$\bord n$' generically to denote any of these and many
other similar possibilities.

Analogous to Definition~\ref{thm:5} we consider representations of~$\bord
n$.  We allow an arbitrary codomain. 

  \begin{definition}[]\label{thm:17}
 Let $\sC$~be a symmetric monoidal $\infn$-category.  An \emph{extended
topological field theory with values in~$\sC$} is a homomorphism $F\:\bord
n\to\sC$.  
  \end{definition}

\noindent
 The homomorphism property means that $F$~respects the $n$~composition laws
as well as the symmetric monoidal structures.  The cobordism hypothesis,
which we take up in the next section, determines the space of
homomorphisms~$F$ in terms of~$\sC$.   
 
For the remainder of this section we indicate some examples which illuminate
the idea of an extended field theory and the flexibility of
Definition~\ref{thm:17}.  

  \begin{example}[]\label{thm:18}
 Let $G$~be a finite group.  Recall from~\S\ref{sec:4} the 2-dimensional
topological field theory $F_G\:\Bord_{\langle 1,2 \rangle}^{SO}\to\CVect$.
In~\eqref{eq:19} we introduced the algebra $A=\Map(G,\CC)$ of functions
on~$G$ under convolution, but only its center made an appearance
in~$F_G$---as~$F_G(\cir)$.  There is an extended field theory~$\hF_G$ of 0-,
1-, and 2-manifolds which has~$\hF_G(\pp)=A$.  The codomain $(\infty
,2)$-category~$\sC$ of any extension has the property that the $(\infty
,1)$-category $\Hom_{\sC}(1,1)$ of endomorphisms of the tensor unit~$1$ is
identified with~$\CVect$.  In fact, $\CVect$~is discrete: objects are complex
vector spaces, 1-morphisms are linear maps, and there are no non-identity
higher morphisms.  So we might hope that $\sC$~is also discrete, an ordinary
2-category.  Furthermore, if $\hF_G(\pp)$~is to be~$A$, then objects of~$\sC$
are algebras.  Thus let $\sC=\CAlg$~be the 2-category whose objects are
complex algebras.  If $A_0,A_1\in \CAlg$, then we define a 1-morphism
$B\:A_0\to A_1$ to be an $(A_1,A_0)$-bimodule~$B$, a complex vector space~$B$
with a left action of~$A_1$ and a right action of~$A_0$.  Composition is by
tensor product over algebras: if $B\:A_0\to A_1$ and $B'\:A_1\to A_2$, then
$B'\circ B\:A_0\to A_2$ is the $(A_2,A_0)$-bimodule $B'\otimes _{A_1}B$.  The
symmetric monoidal structure is given by tensor product over~$\CC$.  The
algebra~$\CC$ is the tensor unit~$1$ and $\Hom_{\CAlg}(1,1)$ is the
collection of $(\CC,\CC)$-bimodules, which is canonically~$\CVect$, as
desired.  A 2-morphism between bimodules is a linear map which intertwines
the algebra actions.  To put this construction in context, we observe that an
\emph{isomorphism} in the 2-category~$\CAlg$ of algebras is a \emph{Morita
equivalence} of algebras.
 
We pause to remark that we have climbed to the next categorical level---from
1-categories to 2-categories---by endowing objects in a 1-category with an
associative unital composition law.  Complex vector spaces form a 1-category,
whereas complex vector spaces which are algebras form a 2-category.  This is
an important general idea, which can be implemented at all categorical levels
and also can be iterated.  For example, if we consider complex vector spaces
with 2~composition laws we obtain a 3-category (of commutative algebras).  We
will meet more examples below.  We can embed $\CAlg$ into the more familiar
2-category of $\CC$-linear categories~$\CCat$: an algebra~$A$ maps to the
linear category of left $A$-modules.  It is usually easier to scale
categorical heights by looking at ``algebra objects'' in an existing
category, rather than by introducing new and more elaborate constructs.  
 
Returning to  
  \begin{equation}\label{eq:32}
     \hF_G\:\bordso 2\longrightarrow \CAlg,
  \end{equation}
once we posit $\hF_G(\pp)=A=\Map(G,\CC)$, we can compute~$\hF_G(\cir)$ as
follows.  We know that $\hF_G(\pt_-)$~is the dual to~$\hF_G(\pp)$, since
$\pp$~and $\pt_-$~are dual in~$\bordso 2$, and it turns out that the dual
algebra is the \emph{opposite} algebra~$A^o$.  The coevaluation in
Figure~\ref{fig:8} is the left $(A\otimes A^o)$-module~$A$, and the
evaluation is the right~$(A^o\otimes A)$-module~$A$.  After permuting the two
boundary points of the evaluation, we compose coevaluation and evaluation to
compute 
  \begin{equation}\label{eq:33}
     \hF_G(\cir) = A\,\otimes \mstrut _{\!\!A\otimes A^o}\;A. 
  \end{equation}
This tensor product is the \emph{Hochschild homology} of the algebra~$A$.  We
can easily compute it explicitly.  Tensoring over~$A$ gives the tensor
product~$A\otimes _AA$ of the right~$A$-module~$A$ with the left
$A$-module~$A$, which is canonically~$A$ by multiplication.  Then the
$A^o$-action is by left and right multiplication, so letting $[A,A]\subset A$
denote the subspace spanned by elements of the
form~$a_1a_2-a_2a_1,\;a_1,a_2\in A$, we conclude $\hF_G(\cir)=A/[A,A]$.  This
is \emph{not} the center of~$A$, which is what we expect from the text
after~\eqref{eq:27}.  To identify the vector space~$A/[A,A]$ with the center
of~$A$ we need one more piece of data, a nondegenerate \emph{trace} $\tau
\:A\to\CC$ on~$A$.  Nondegeneracy means that $a_1,a_2\mapsto \tau (a_1a_2)$
is a nondegenerate pairing, and then we identify the quotient~$A/[A,A]$ with
the orthogonal subspace~$[A,A]^\perp\subset A$, which is easily identified
with the center of~$A$.  The pair~$(A,\tau )$ is a \emph{Frobenius algebra}.
For $A=\Map(G,\CC)$ we use the trace~\eqref{eq:20}.
 
The cobordism hypothesis, stated for framed manifolds in
Theorem~\ref{thm:14}, asserts that $\hF_G$~is determined by its value
on~$\pp$.  This is true here, but `value on~$\pp$' must be interpreted as the
pair~$(A,\tau )$.  The extra datum~$\tau $ is necessary as $\hF_G$~is an
oriented theory, not simply a framed theory; see Theorem~\ref{thm:28} and
Example~\ref{thm:29}. 
 
In~\S\ref{sec:4} we described an approach to the non-extended theory~$F_G$
using a finite version of the path integral in physics, which amounts to
counting principal $G$-bundles.  The finite path integral extends to give an
\emph{a priori} construction of~$\hF_G$ in which $\hF_G(\pp)=A$~is the result
of a computation; see~\cite{F,FHLT} for details.
  \end{example}

  \begin{example}[]\label{thm:20}
 Historically, 3-dimensional Chern-Simons theory~\cite{W3} was the example
which most pointed the way towards extended topological field theories.  The
approach of Reshetikhin-Turaev~\cite{RT1,RT2} to the resulting invariants of
3-manifolds and links begins with a quantum group, in the form of a complex
linear category with extra structure, a \emph{modular tensor
category}~\cite{MSei}.  By contrast, Witten begins with the Chern-Simons
functional and uses the path integral.  The relationship between the
approaches, worked out in~\cite{F} for finite gauge groups, is that
Chern-Simons is a (partially) extended theory of 1-, 2-, and~3-manifolds
whose value on~$\cir$ is the modular tensor category.  A complete
construction of this 1-2-3 theory beginning from quantum group data was given
in~\cite{Tu}; see also~\cite{Wa}.  There is current work, for
example~\cite{BDH}, to construct a fully extended 0-1-2-3 theory.
  \end{example}

  \begin{example}[]\label{thm:21}
 The previous two examples are discrete: there are no interesting invariants
for families of manifolds beyond those for single manifolds.  That an
extension of Definition~\ref{thm:5} to families would be fruitful emerged in
the 1990s from 2-dimensional field theories.  Segal promoted the idea of a
cochain-valued topological field theory~\cite{Se3}, and there were several
mathematical works which pointed towards invariants for families of
manifolds; a quirky sample is~\cite{LZ,G,KM,BC}.  The most definitive work in
this direction is by Kevin Costello~\cite{Co}, who constructed a theory of
``open-closed'' topological 2-dimensional field theories in families from
Calabi-Yau categories.  These are closely related to fully extended
2-dimensional theories; see~\cite[\S4.2]{L1}. 
  \end{example}

  \begin{example}[]\label{thm:22}
 Another motivating example for the cobordism hypothesis which includes
invariants for families of manifolds is \emph{string topology}, which defines
invariants of compact manifolds using its loop space and Riemann surfaces.
It was introduced by Chas-Sullivan~\cite{CS}, and there is a large literature
which follows.  See~\cite[\S4.2]{L1} for the relation with the cobordism
hypothesis.  
  \end{example}

   \section{The cobordism hypothesis}\label{sec:6}

Recall from~\S\ref{sec:4} that objects in the image of a non-extended
topological field theory obey a finiteness condition, expressed in
categorical terms by dualizability (Definition~\ref{thm:7}).  There is an
analogous finiteness condition called \emph{adjointability} for
$k$-morphisms, $1\le k\le n-1$, in an extended $n$-dimensional field theory.
We give the definition for 1-morphisms, which specializes to the traditional
notion of \emph{adjoint functors} in category theory~\cite{Ka} for the
2-category of categories.

  \begin{definition}[]\label{thm:23}
 Let $\sC$ be a 2-category; $x,y\in \sC$ objects in~$\sC$; and suppose
$f\:x\to y$, $g\:y\to x$ are 1-morphisms.  Then $f$~is a \emph{left adjoint}
to~$g$ if there exist 2-morphisms $u\:\id_x\Rightarrow g\circ f$ and
$c\:f\circ g\Rightarrow\id_y$ such that the compositions
  \begin{equation}\label{eq:34}
     f = f\circ \id_x \xRightarrow{\id\times u}f\circ g\circ f
     \xRightarrow{c\times \id}\id_y\circ f=f 
  \end{equation}
and 
  \begin{equation}\label{eq:35}
     g = id_x\circ g \xRightarrow{u\times \id}g\circ f\circ g
     \xRightarrow{\id\times c}g\circ \id_y=g
  \end{equation}
are identity 2-morphisms. 
  \end{definition}

\noindent
 We then say that $g$~is a right adjoint to~$f$, and $u$, $c$~are the unit
and counit of an adjunction.  The compositions~\eqref{eq:34}
and~\eqref{eq:35} are the 2-morphism version of the S-diagram
compositions~\eqref{eq:13}.  The corresponding definition for
$\infn$-categories and higher morphisms is similar, but the compositions are
only the identity maps up to higher morphisms, or equivalently are identity
maps in a homotopy category which remembers higher morphisms only up to
equivalence.  \emph{Invertible} maps have adjoints---the inverse \emph{is} an
adjoint---so adjointability is weaker than invertibility.   

  \begin{remark}[]\label{thm:31}
 If an $n$-morphism in an $n$-category, or $\infn$-category, is adjointable
then it is invertible.  This follows since the unit and counit of an
adjunction, which are $(n+1)$-morphisms, are invertible.
  \end{remark}

Let $\sC$~be a symmetric monoidal $\infn$-category and $F\:\bord n\to\sC$ an
extended field theory.  Then just as $F(\pt)$~is dualizable, so too is
$F(W)$~adjointable for every $k$-dimensional bordism~$W$ with $1\le k\le
n-1$.  This is an extended finiteness condition satisfied by an extended
topological field theory.  We extract from~$\sC$ all objects which have
duals, and whose duality data have adjoints, which in turn have adjoints,
etc.

  \begin{lemma}\cite[\S2.3]{L1}\label{thm:24}
 Let $\sC$~be a symmetric monoidal $\infn$-category.  There is an
$\infn$-category~$\sCfd$ and a homomorphism $i\:\sCfd\to\sC$ so that (i)\
every object in~$\sCfd$ is dualizable and every $k$-morphism, $1\le k\le n-1$,
is adjointable, and (ii)\ $i\:\sCfd\to\sC$ is universal with respect to~(i).
  \end{lemma}

\noindent
 Here `fd' stands either for `fully dualizable' or `finite dimensional'.  An
$\infn$-category which satisfies~(i) is said to ``have duals'', as in the
statement of Theorem~\ref{thm:13}.  The finiteness condition on a topological
field theory $F\:\bord n\to\sC$ may be summarized by the diagram
  \begin{equation}\label{eq:36}
     \xymatrix@R-15pt{&\sCfd\ar[dd]^i\\ \bord n\ar@{-->}[ur]\ar[dr]^F\\&\sC} 
  \end{equation}
In other words, $F$~factors through~$\sCfd$.

Extended topological field theories $F\:\bord n\to\sC$ are the objects of an
$\infn$-category we denote $\Hom(\bord n,\sC)$.  A 1-morphism $\eta \:F_0\to
F_1$ between two homomorphisms assigns a $(k+1)$-dimensional morphism~$\eta
(W)\:F_0(W)\to F_1(W)$ to each $k$-dimensional bordism~$W$.  The fact that
adjointable $n$-morphisms are invertible (Remark~\ref{thm:31}) implies, after
some argument, that any 1-morphism~$\eta $ is in fact an isomorphism.  The
same applies to higher morphisms.  It follows that $\Hom(\bord n,\sC)$ is in
fact an $(\infty,0)$-category---all morphisms are invertible---so according
to Remark~\ref{thm:19} can be viewed as a \emph{space}.  In other words, the
collection of extended topological field theories with values in~$\sC$ is a
space.
 
The cobordism hypothesis identifies the space~$\Hom(\bord n,\sC)$ with a
space constructed directly from~$\sC$ by combining Lemma~\ref{thm:24} with
another universal construction.

  \begin{lemma}\cite[\S2.4]{L1}\label{thm:25}
 Let $\sD$ be an $\infn$-category.  There is an $\infty$-groupoid~$\sDt$
and a homomorphism $j\:\sDt\to\sD$ so that (i)\ every $k$-morphism, $k>0$,
in~$\sDt$ is invertible, and (ii)\ $j\:\sDt\to\sD$ is universal with respect
to~(i).
  \end{lemma}

\noindent
 The $\infty $-groupoid~ $\sDt$, which is an $\infty $-category in which
every morphism is invertible, may be constructed from~$\sD$ by removing all
noninvertible morphisms.
 
Finally, we can state a precise version of the cobordism hypothesis, first
for $n$-framed manifolds. 

  \begin{theorem}[Cobordism hypothesis: framed version]\label{thm:26}
 Let $\sC$~be a symmetric monoidal $\infn$-category.  Then the map 
  \begin{equation}\label{eq:37}
     \begin{aligned} \Hom(\bordfr n,\sC)&\longrightarrow \sCfdt \\
      F&\longmapsto F(\pp)\end{aligned} 
  \end{equation}
is a homotopy equivalence of spaces. 
  \end{theorem}

\noindent
 At this point the reader should refer back to the heuristic versions stated
in~\S\ref{sec:1} as well as the discrete 1-dimensional version in
Theorem~\ref{thm:9}.  In particular, the cobordism hypothesis is a theorem
about smooth manifolds and their diffeomorphism groups, which is reflected by
the method of proof.
 
Suppose $W$~is a bordism of dimension~$k\le n$ which is $n$-framed.  Recall
that the $n$-framing is an isomorphism $(n)\to (n-k)\oplus TW$, where
$(j)$~denotes the trivial real vector bundle of rank~$j$ over~$W$.  The
orthogonal group\footnote{It is perhaps more natural to use the full general
linear group~$GL(n;\RR)$, but all of the topological information is carried
by the maximal compact subgroup $O(n)\subset GL(n;\RR)$.}~$O(n)$ acts on
framings by precomposition with constant orthogonal maps $(n)\to (n)$.  This
induces an action of~$O(n)$ on the space $\Hom(\bordfr n,\sC)$.

  \begin{corollary}[]\label{thm:27}
 There is a canonical action of the orthogonal group~$O(n)$ on the
space~$\sCfdt$.
  \end{corollary}

Let $G$~be a Lie group equipped with a homomorphism $\rho \:G\to O(n)$.  A
\emph{$G$-structure} on a bordism~$W$ is a reduction of structure group of
its tangent bundle to~$G$ along~$\rho $.  More precisely, choose a Riemannian
metric on~$W$ (this is a contractible choice).  Then a $G$-structure is a
principal $G$-bundle $P\to W$ together with an isomorphism of the associated
$G$-bundle~$\rho (P)$ with the bundle of orthonormal frames of~$(n-k)\oplus
TW $.  For example, for $G=\{e\}$ a $G$-structure is an $n$-framing, and for
$G=SO(n)$ it is an orientation.  There is a bordism category~$\bordG n$ of
manifolds with $G$-structure.

  \begin{theorem}[Cobordism hypothesis: $G$-structure version]\label{thm:28}
 The map 
  \begin{equation}\label{eq:38}
     \begin{aligned} \Hom(\bordG n,\sC)&\longrightarrow
     \bigl(\sCfdt\bigr)^{hG} \\ 
      F&\longmapsto F(\pp)\end{aligned} 
  \end{equation}
is a homotopy equivalence between the space of extended topological field
theories on $G$-manifolds and the homotopy fixed point space of the
$G$-action on~$\sCfdt$. 
  \end{theorem}

\noindent
 Here $G$~acts through the homomorphism $\rho \:G\to O(n)$ and the
$O(n)$-action given in Corollary~\ref{thm:27}. 

  \begin{example}[]\label{thm:29}
 For $n=2$ an oriented 2-dimensional theory is determined by the value
on~$\pp$, but in the fixed point space.  Consider $\sC=\CAlg$, as in
Example~\ref{thm:18}.  First, the 2-category~$\CAlg^{\textnormal{fd}}$ of
fully dualizable complex algebras has objects finite dimensional semisimple
algebras, i.e., finite products of matrix complex algebras.  (A proof may be
found in~\cite[\S3.2]{Da}.)  A point in the \emph{homotopy} fixed point space
of the $SO(2)$-action includes extra \emph{data}---in this case being a fixed
point is not a \emph{condition}---and the extra data here is the
nondegenerate trace~$\tau $ discussed in Example~\ref{thm:18};
see~\cite[Example~2.8]{FHLT} for details.
  \end{example}

We are not going to attempt to summarize the proof sketched in~\cite{L1} in
any detail.  Rather, we give a very rough intuition for why the cobordism
hypothesis might be true.  Our exposition in~\S\ref{sec:4}, in particular the
proof of Theorem~\ref{thm:9}, emphasizes the role of Morse theory.  The
existence of Morse functions allows the decomposition of a bordism into a
composition of elementary bordisms~\eqref{eq:31}.  These elementary bordisms
encode the evaluations and coevaluations, or units and counits, of duality
and adjointness data.  That is clear in the proof of Theorem~\ref{thm:9}.  As
another example, Figure~\ref{fig:15} may be read as a counit for the
adjunction between the two 1-morphisms coev, ev in Figure~\ref{fig:8}.  So if
$x\in \sC$ is fully dualizable, a choice of duality data---duals and adjoints
all the way up---defines~$F$ on elementary bordisms.  As arbitrary bordisms
are compositions of elementary bordisms, $F$~can be extended to arbitrary
bordisms.  In other words a Morse function gives, in principle, a way to
evaluate~$F(W)$ for every bordism~$W$.  The issue is whether $F(W)$~ is
well-defined.  The duality data involves choices, and we must be sure that
those choices can be made coherently.  This is expressed via contractibility
statements.  The first is that the space of duality data for a dualizable
object~$x$ is contractible.  The second generalizes the connectivity
statement at the heart of Cerf theory~\cite{C}.  Lurie uses a higher
connectivity theorem of Kiyoshi Igusa~\cite{I} for the space of
\emph{generalized framed Morse functions}.  Such functions relax the
nondegeneracy condition at a critical point to allow a single degeneracy, as
in ~\eqref{eq:3}, and also include a framing of the negative definite
subspace at a critical point.  Igusa proves that on a $k$-dimensional
manifold this space is $k$-connected.\footnote{It is in fact a consequence of
the cobordism hypothesis that this space of functions is weakly contractible.
This has been proved independently of the cobordism hypothesis in~\cite{EM}
and also in unpublished work of Galatius.}  These contractibility statements
are central to the proof, but it is a highly nontrivial problem to organize
the higher categorical data to apply these theorems.  The solution to that
problem, described in detail in~\cite{L1}, is equally central to the proof.

   \section{Implications, extensions, and applications}\label{sec:7}

Some brief vignettes illustrate the scope of the extended topological field
theory and the cobordism hypothesis.

  \subsection*{Invertible theories and Madsen-Tillmann spectra}

Recall from Lemma~\ref{thm:25} that any $\infn$-category~$\sD$ has an
underlying $\infty$-groupoid~$\sDt\to\sD$, which may be identified with a
space.  There is a quotient construction as well.

  \begin{lemma}[]\label{thm:33}
  Let $\sD$ be an $\infn$-category.  There is an $\infty$-groupoid~$\bsD$
and a homomorphism $q\:\sD\to\bsD$ so that (i)\ every $k$-morphism, $k>0$,
in~$\bsD$ is invertible, and (ii)\ $q\:\sD\to\bsD$ is universal with respect
to~(i). 
  \end{lemma}

These constructions are relevant to \emph{invertible topological field
theories}.  We say an object~$x$ in a symmetric monoidal $\infn$-category is
\emph{invertible} if it has an inverse~$y$ for the monoidal structure:
$x\otimes y$~is isomorphic to the unit object.

  \begin{definition}[]\label{thm:32}
 A topological field theory $\alpha \:\bord n\to\sC$ is \emph{invertible} if
$\alpha (W)$~is invertible for all objects and morphisms~$W$. 
  \end{definition}

\noindent
 It follows from the cobordism hypothesis that $\alpha $~is invertible if and
only if $\alpha (\pp)$~is invertible.  By the universal properties an
invertible field theory $\alpha \:\bord n\to\sC$ factors through~$|\bord n|$
and~$\sCfdt$:
  \begin{equation}\label{eq:39}
     \xymatrix{\bord n\ar[r]^(.6)\alpha \ar[d]_q&\sC\\ |\bord
     n|\ar[r]^(.55){\tilde\alpha }& \sCfdt\ar[u]_j} 
  \end{equation}
Since $\bord n$~and $\sC$~are symmetric monoidal, so too are~$|\bord n|$
and~$\sCfdt$.  An $\infty$-groupoid is equivalent to a space
(Remark~\ref{thm:19}), and a symmetric monoidal $\infty$-groupoid is
equivalent to an infinite loop space, i.e., the 0-space of a spectrum.
Furthermore, $\tilde\alpha $~is an infinite loop space map.  This reduces the
study of invertible topological field theories to a problem in stable
homotopy theory.

  \begin{remark}[]\label{thm:35}
 Invertible field theories play a role in ordinary quantum field theory, for
example as \emph{anomalies}. 
  \end{remark}

A corollary of the cobordism hypothesis~\cite[\S2.5]{L1} determines the
homotopy type of the spectrum $|\bord n|$.  Consider first the bordism
$\infn$-category $\bordfr n$ of $n$-framed manifolds.  The cobordism
hypothesis, in the heuristic form Theorem~\ref{thm:13}, asserts that $\bordfr
n$ is free on one generator.  It follows that so too is $|\bordfr n|$.  The
latter is a spectrum, and the free spectrum on one generator is the sphere
spectrum.  For the bordism $\infn$-category of $G$-manifolds $\bordG n$ the
cobordism hypothesis in the form Theorem~\ref{thm:28} implies that $|\bordG
n|$ is the $n^{\textnormal{th}}$~ suspension of a Madsen-Tillmann spectrum.
(These spectra are mentioned in~\S\ref{sec:2} before Definition~\ref{thm:4}.)

An $\infty$-groupoid---or $(\infty ,0)$-category---is a model for a space.
We may view an $\infn$-category as a generalization of a space which allows
noninvertibility.  From that perspective the cobordism hypothesis is a
generalization of the Madsen-Weiss theorem.

  \subsection*{Variations on the cobordism hypothesis}

Morrison and Walker~\cite{MoW} take a somewhat different, but closely
related, approach to extended topological field theories which incorporates
dualizability from the beginning.

In~\cite[\S4]{L1} Lurie describes several applications and extensions of the
cobordism hypothesis.  One important extension is to manifolds with
singularities, though there are many special cases which do not in fact
involve singularities.  To illustrate, in Example~\ref{thm:18} we described a
2-dimensional oriented field theory~$F$ associated to a Frobenius
algebra~$A$.  Now suppose that $M$~is a left $A$-module.  Recall that
$M$~determines a 1-morphism $M\:1\to A$ in the Morita 2-category of algebras,
where the tensor unit~$1$ is the trivial algebra~$\CC$.  We might ask what
sort of field theory we can associate to the pair~$(A,M)$, assuming
sufficient finiteness..  A physicist might describe~$M$ as giving a
\emph{boundary condition} for~$F$, and so extend~$F$ to a field theory~$\tF$
in which some boundaries are ``colored'' with the boundary condition~$M$.
For example, a closed interval with one endpoint colored is associated to~$M$
as a left $A$-module; the closed interval with both endpoints colored is
associated to~$M$ as a vector space.  The coloring represents a coning off of
a point, which is viewed as a manifold with singularities.  This is just the
tip of the iceberg of possibilities opened up by the cobordism hypothesis
with singularities.
 
From the point of view of algebra, given that $\bordfr n$ is the free
symmetric monoidal $\infn$-category with duals on one generator, we might ask
how to describe more general symmetric monoidal $\infn$-categories specified
by generators and relations.  Roughly speaking, the cobordism hypothesis with
singularities identifies these as bordism categories of manifolds with
singularities.

  \subsection*{Applications to topology}

We indicated briefly in Example~\ref{thm:20} the important role that
Chern-Simons theory played in the development of extended topological quantum
field theories.  That theory encodes invariants of 3-manifolds and links.
Newer invariants of links and low dimensional manifolds were in part inspired
by notions in extended field theory.  Crane and Frenkel~\cite{CF} suggested
that ``categorification'' of the 3-dimensional invariants would lead to new
invariants, potentially related to Donaldson invariants.  Later
Khovanov~\cite{Kh} introduced such a categorification of the Jones
polynomial.  This now has a proposed derivation from quantum field
theory~\cite{GSV,W4}.

There is current research in many directions which will potentially take
advantage of more powerful aspects of extended field theories and the
cobordism hypothesis in contexts which are not discrete and semisimple.  For
example, the cobordism hypothesis illuminates string topology invariants and
topological versions of Hochschild homology and its cousins~\cite{BCT}.  It
also appears in several discussions of the 2-dimensional extended topological
field theories relevant for mirror symmetry: the ``A-model'' and the
``B-model''.  There is an enormous literature on this subject; see~\cite{Te2}
for one recent example which uses ideas around the cobordism hypothesis.

  \subsection*{Applications to algebra}

Now we shift focus from topology and bordism categories to the
codomain~$\sC$.  Quite generally a homomorphism in algebra organizes the
codomain according to the structure of the domain.  This principle is often
applied in the context of group actions on sets, for example: the structure
of orbits and stabilizers illuminates the situation at hand.  Here if
$F\:\bord n\to\sC$ is a homomorphism, and $F(\pp)=x$ then we can study~$x$
using smooth manifolds and their gluings.
 
One application is to \emph{$E_k$-algebras}, which are objects in a symmetric
monoidal category which have $k$~associative composition laws.  We met
$E_1$-algebras (ordinary associative algebras) in the category~$\CVect$ of
complex vector spaces in Example~\ref{thm:18}.  An $E_2$-algebra in~$\CVect$
is a commutative algebra and there is nothing higher up: an $E_k$-algebra
for~ $k>2$ is also a commutative algebra.  More interesting examples are
obtained if we look in other symmetric monoidal categories, for example the
$\infty $-category of chain complexes.  In~\cite[\S4.1]{L1} Lurie describes
some relationships between the cobordism hypothesis and $E_k$-algebras in
$\infn$-categories.  In particular, an $E_k$-algebra~$A$ in an
$\infn$-category~$\sC$ is automatically $k$-dualizable, so determines a
homomorphism $F\:\bordfr k\to E_k(\sC)$, where $E_k(\sC)$~is the $(\infty
,n+k)$-category whose objects are $E_k$-algebras in~$\sC$.  Thus
$E_k$-algebras may be studied with smooth manifolds.  For example, if $A$~is
an ordinary algebra ($E_1$-algebra), then in the associated field theory
$F(\cir)$~is the Hochschild homology of~$A$ (see~\eqref{eq:33} for a simple
example).  Since the circle is an $E_2$-algebra in the bordism category, so
too is the Hochschild homology~$F(\cir)$.  This assertion is the Deligne
conjecture, which together with generalizations is proved in many works, for
example~\cite{Co,KS,L2,BFN}.  (We remark that there are several other proofs
of the Deligne conjecture.)
 
As another application of the cobordism hypothesis to algebra, we mention
ongoing work~\cite{DSS} which proves that a \emph{fusion category}~\cite{ENO}
is 3-dualizable.  A fusion category is a special type of \emph{tensor
category}, and a tensor category is an $E_1$-algebra in the 2-category of
linear categories.  So tensor categories form a 3-category, and it is in that
3-category that fusion categories are fully dualizable.  The associated
3-dimensional framed field theory can be brought to bear on the study of
fusion categories.  We remark that simple topological diagrams involving
0-~and 1-dimensional manifolds are usually used to study fusion categories
and their cousins.  The cobordism hypothesis opens up the possibility of
using the more powerful topology of 3-dimensional manifolds.  In related
ongoing work of the author and Teleman, we consider $E_2$-algebras in the
2-category of linear categories; they comprise the 4-category of
\emph{braided tensor categories}.  We prove that \emph{modular} tensor
categories are invertible, which now gives a 4-dimensional perspective on
quantum groups.

  \subsection*{Applications to representation theory}

In~\S\ref{sec:4} and in Example~\ref{thm:18} we illustrated a very simple,
discrete 2-dimensional field theory associated to a finite group~$G$.  There
is also a 3-dimensional field theory with values in the 3-category of tensor
categories; it attaches the tensor category of vector bundles over~$G$ under
convolution to~$\pt$.  (The theory is unoriented---as is the 2-dimensional
theory---so we have an unframed unoriented unadorned point.)  That theory may
be viewed as the simplest case of 3-dimensional Chern-Simons theory
(Example~\ref{thm:20}).  Ben-Zvi and Nadler~\cite{BN} study the analogous
theory for a reductive complex group~$G$.  Discrete categories are futile
here; the full force of $\infty $-categories comes into play.  One would like
a 3-dimensional theory which generalizes that of a finite group, and now
attaches the symmetric monoidal $\infty $-category of $\sD$-modules on~$G$ to
a point.  However, the necessary finiteness conditions are not satisfied.
Instead, they construct a related 2-dimensional field theory, the
\emph{character theory}, which assigns to a point the \emph{Hecke category}
associated to~$G$.  Then one computes that the category of Lusztig's
\emph{character sheaves} is attached to~$\cir$.  The character theory may be
viewed as a dimensional reduction of a 4-dimensional field theory~\cite{KW}
related to the geometric Langlands program.  It seems likely that the
topological field theory perspective, and the cobordism hypothesis, will shed
light on old questions in the representation theory of semisimple Lie groups.

  \subsection*{Echos in quantum field theory}

As mentioned earlier, quantum field theorists traditionally only studied
\emph{2-tier} theories: correlation functions on $n$-manifolds and Hilbert
spaces attached to~$(n-1)$-manifolds.  In recent years the ideas
mathematicians have developed around extended field theories, including the
cobordism hypothesis, have seeped into physics.  In 2-dimensional conformal
field theory there is a category of boundary conditions, called
\emph{D-branes}, and in topological versions this is understood to be part of
an extended field theory.  Higher dimensional analogs are now common;
see~\cite{Kap} for a recent review.  For example, Kapustin-Witten~\cite{KW}
study a topological twist of the 4-dimensional $N=4$~ supersymmetric
Yang-Mills theory.  Going beyond the traditional two tiers, this theory
attaches a category to every closed 2-manifold.  Kapustin-Witten relate that
to a category which appears in the geometric Langlands program.  The story is
richer: there is a family of theories parametrized by~$\CP^1$ and
\emph{$S$-duality} acts as an involution on the theories.  This suggests an
equivalence between two different categories attached to a 2-manifold, which
is a topological version of the basic conjecture in the geometric Langlands
program.
 
The maximally supersymmetric $N=4$~Yang-Mills theory is the dimensional
reduction of a 6-dimensional supersymmetric field theory which has
superconformal invariance.  Its name `the (2,0) superconformal field theory
in six dimensions' reflects its symmetry group; a simpler name is
`Theory~$\sX$'.  This theory has no classical description.  It is predicted
to exist from limiting arguments in string theory.  Its mysterious nature
justifies the appellation `Theory~$\sX$', as does its dimension: si$\sX$.  A
few properties can be predicted from string theory, and these can be used to
study compactifications to lower dimensions.  Among the many protagonists
here we mention Gaiotto~\cite{Ga} and Gaiotto-Moore-Neitzke~\cite{GMN}.  One
important idea---which is clearly inspired by extended field theory and the
activity surrounding the cobordism hypothesis---is to study compactifications
of the 6-dimensional theory as a function of the compactifying manifold.
This is formalized as follows.  Suppose $F\:\bord 6\to\sC$ is a 6-dimensional
extended topological theory.  Then for any closed 2-manifold~$N$ we obtain a
4-dimensional theory~$F_N\:\bord 4\to\Hom_{\sC}\bigl(F(N),F(N) \bigr)$
defined using Cartesian product:\footnote{The bordism groups of Pontrjagin
and Thom are \emph{rings} with multiplication given by Cartesian product.
Our discussion of topological field theory has not used this ring structure
until now.}  $F_N(M)=F(N\times M)$.  Now view~$F_N$ as a function of~$N$.
Then we obtain a 2-dimensional extended field theory with values in the
$(\infty ,4)$-category of 4-dimensional field theories!  The flexibility in
Definition~\ref{thm:17} which allows arbitrary codomains is heavily used
here.  One can get other field theories by composing with homomorphisms out
of 4-dimensional theories.  A recent paper~\cite{MoT} implements this idea in
a physics context, and predicts the existence of certain holomorphic
symplectic manifolds.
 
Finally, the renewed interest in $E_n$-algebras and their role in extended
topological field theories may bring some fresh perspectives to quantum field
theories which are not topological.  One axiomatic approach to quantum field
theory~\cite{H} assigns operator algebras to open sets and describes how they
fit together.  This idea was imported in an algebro-geometric framework in
certain mathematical approaches to 2-dimensional conformal field theory, in
vertex operator algebras~\cite{Bo} and chiral algebras~\cite{BeDr}.  These
ideas are circling back to general quantum field theories~\cite{CG} with
potential to shed new light on their structure.

\bigskip\bigskip

These are only a few examples of the potential that extended topological
field theories and the cobordism hypothesis hold in both mathematics and
physics.

 \bigskip\bigskip
\providecommand{\bysame}{\leavevmode\hbox to3em{\hrulefill}\thinspace}
\providecommand{\MR}{\relax\ifhmode\unskip\space\fi MR }
\providecommand{\MRhref}[2]{%
  \href{http://www.ams.org/mathscinet-getitem?mr=#1}{#2}
}
\providecommand{\href}[2]{#2}

  \end{document}